\documentclass[10pt]{article}

\usepackage{amsfonts}
\usepackage{amsmath}   
\usepackage{amssymb}
\usepackage{amsthm}
\usepackage{mathtools}
\usepackage{tikz-cd}
\usepackage{xcolor}
\usepackage[colorlinks=true]{hyperref}
 
\newtheorem*{Th*}{Theorem}
\newtheorem{Th}{Theorem}[section]
\newtheorem{Prop}{Proposition}[section]   
\newtheorem{Lem}{Lemma}[section]   
\newtheorem{Coro}{Corollary}[section]   

\newtheorem{Rem}{Remark}[section]

\newcommand{\R}{\mathbb{R}}
\newcommand{\Z}{\mathbb{Z}}
\newcommand{\C}{\mathbb{C}}
\newcommand{\T}{\mathbb{T}}
\newcommand{\LL}{\mathcal{L}}

\newcommand{\D}{\mathcal{D}}
\newcommand{\RR}{\mathcal{R}}
\newcommand{\hu}{{\widehat u}}
\newcommand{\h}{\mathfrak{h}}
\newcommand{\1}{\langle} 
\newcommand{\2}{\rangle} 

\newcommand{\tb}[2]{\genfrac{}{}{0pt}{}{#1}{#2}} 

\newcommand{\sign}{\mathop{\rm sign}}

\newcommand{\p}{{\partial}}
\newcommand{\HH}{\mathcal{H}}
\newcommand{\spec}{\mathop{\rm Spec}\nolimits}

\newcommand{\re}{\mathop{\rm Re}}

\newcommand{\Wert}{\mathop{\rm Vert}\nolimits}
\newcommand{\bcdot}{\boldsymbol{\cdot}}

\begin{document}

\title{On the analytic Birkhoff normal form of the Benjamin-Ono equation and applications}   
 
\author{P. G\'erard, T. Kappeler\footnote{T.K.  partially supported by the Swiss National Science Foundation.} 
\,and P. Topalov\footnote{P.T.  partially supported by the Simons Foundation,  Award \#526907.}}
  
\maketitle

\begin{abstract}  
In this paper we prove that the Benjamin-Ono equation admits an analytic Birkhoff normal form in an open neighborhood of 
zero in $H^{s}_{0}(\T, \R)$ for any $s>-1/2$ where $H^{s}_{0}(\T, \R)$ denotes the subspace of the Sobolev space 
$H^{s}(\T, \R)$ of elements with mean $0$. As an application we show that for any $-1/2<s<0$, the flow map of 
the Benjamin-Ono equation $\mathcal{S}_0^t : H^{s}_{0}(\T, \R)\to H^{s}_{0}(\T, \R)$ is {\em nowhere locally 
uniformly continuous} in a neighborhood of zero in $H^{s}_{0}(\T, \R)$.
\end{abstract}    

\smallskip


\smallskip


\tableofcontents

\section{Introduction}\label{Introduction}
In this paper we study the Benjamin-Ono equation on the torus $\T := \R/2\pi\Z$,
\begin{equation}\label{eq:BO}
\partial_t u=\partial_x\big(|\partial_x|u-u^2\big), 
\end{equation}
where $u\equiv u(x,t)$, $x\in\T$, $t\in\R$, is real valued and $|\partial_x| : H^\beta_c\to H^{\beta-1}_c$, $\beta\in\R$, 
is the Fourier multiplier
\[
|\partial_x| : \sum_{n\in\Z}\widehat{v}(n) e^{i n x}\mapsto\sum_{n\in\Z}|n|\,\widehat{v}(n) e^{i n x}
\]
where $\widehat{v}(n)$, $n\in\Z$, are the Fourier coefficients of $v\in H^\beta_c$ and $H^\beta_c\equiv H^\beta(\T,\C)$
is the Sobolev space of complex valued distributions on the torus $\T$. 
The equation \eqref{eq:BO} was introduced in 1967 by Benjamin \cite{Ben1967} and Davis $\&$ Acrivos \cite{DA1967} as 
a model for a special regime of internal gravity waves at the interface of two fluids. It is well known that \eqref{eq:BO} admits 
a Lax pair representation (cf. \cite{Na1979}) that leads to an infinite sequence of conserved quantities (cf. \cite{Na1979}, \cite{BK1979})
and that it can be written in Hamiltonian form  with Hamiltonian
\begin{equation}\label{eq:BO-Hamiltonian}
\HH(u):=\frac{1}{2\pi}\int_0^{2\pi}\Big(\frac{1}{2}\big(|\p_x|^{1/2}u\big)^2-\frac{1}{3}u^3\Big)\,dx
\end{equation}
by the use of the Gardner bracket
\begin{equation}\label{eq:bracket}
\{F,G\}(u):=\frac{1}{2\pi}\int_0^{2\pi}\big(\partial_x\nabla_u F\big)\nabla_u G\,dx
\end{equation}
where $\nabla_u F$ and $\nabla_u G$ are the $L^2$-gradients of $F,G\in C^1(H^s_r,\R)$ at $u\in H^s_r$ where $H^s_r\equiv H^s(\T,\R)$
is the Sobolev space of real valued distributions on $\T$.
By the Sobolev embedding $H^{1/2}_r\hookrightarrow L^3(\T,\R)$, the Hamiltonian \eqref{eq:BO-Hamiltonian} is 
well defined and real analytic on $H^{1/2}_r$, the {\em energy space} of \eqref{eq:BO}.
The problem of the existence and the uniqueness of the solutions of the Benjamin-Ono equation is well studied 
-- see \cite{GKT1}, \cite{Saut2019} and references therein. 
We refer to \cite{Saut2019} for an excellent survey and a derivation of \eqref{eq:BO}.

By using the Hamiltonian formalism for \eqref{eq:BO}, it was recently proven in \cite{GK,GKT1} that for any $s > -1/2$,
the Benjamin-Ono equation has a Birkhoff map
\begin{equation}\label{eq:Phi-introduction}
\Phi : H^{s}_{r,0}\to\h^{\frac{1}{2}+s}_{r,0} , \ u \mapsto \big( (\overline{\Phi_{-n}(u)}  )_{n \le -1}, (\Phi_n(u))_{n \ge 1} \big) ,
\end{equation}
where
\[
H^\beta_{r,0}:= H^\beta_r \cap H^\beta_{c,0}, \qquad
H^\beta_{c,0}:=\big\{u\in H^\beta_c\,\big| \, \hu(0)=0\big\} ,
\]
and
\begin{equation}\label{eq:h_{r,0}}
\h^\beta_{r,0}:=\big\{z\in\h^\beta_{c, 0}\,\big| \ z_{-n}=\overline{z}_n \ \forall n\ge 1\big\}
\end{equation}
is a real subspace in the complex Hilbert space 
\begin{equation}\label{eq:h_{c,0}}
\h^\beta_{c,0}:=\big\{ (z_n)_{n\in\Z}\,\big| \, z_0=0\,\,\,\text{\rm and}\,\,\,
\sum_{n\in\Z}\1 n\2^{2\beta}|z_n|^2 < \infty\big\},\quad \1 n \2 := \max\{1,|n|\}.
\end{equation}
By \cite[Theorem 1]{GK} and \cite[Theorem 6]{GKT1}, the Birkhoff map \eqref{eq:Phi-introduction} is a {\em homeomorphism} that 
transforms the trajectories of the Benjamin-Ono equation \eqref{eq:BO} into straight lines, which are winding around
the underlying invariant torus, defined by the action variables.
Furthermore, these trajectories evolve 
on the isospectral sets of potentials of the corresponding Lax operator (see \eqref{eq:L} below). 
In this sense, the Birkhoff map can be considered as a non-linear Fourier transform that significantly simplifies 
the construction of solutions of \eqref{eq:BO}. This fact allows us to prove
that for any $-1/2<s<0$,  \eqref{eq:BO} is globally $C^0$-well-posed on $H^{s}_{r,0}$ 
(\cite[Theorem 1]{GKT1}) improving in this way the previously known well-posedness results (see \cite{Mo,MoP}). 
This result is sharp by \cite[Theorem 2]{GKT1}. Additional applications of the Birkhoff map include the proof of 
the almost periodicity of the solutions of the Benjamin-Ono equation and the orbital stability of the Benjamin-Ono traveling waves
(see \cite[Theorem 3 and Theorem 4]{GKT1}).
Finally,  since the Hamiltonian \eqref{eq:BO-Hamiltonian} is well defined on $H^s_r$ for any $s\ge 1/2$,
it follows from \cite[Proposition 8.1]{GK} that for any $z$ in $\h^{\frac{1}{2}+s}_{r,0}$, $s \ge 1/2$,
\begin{equation}\label{eq:BO-Hamiltonian_in_BNF}
\mathcal{H}\circ\Phi^{-1}(z)=\sum_{n=1}^\infty n^2 I_n-\sum_{n=1}^\infty\Big(\sum_{k=n}^\infty I_k\Big)^2,
\end{equation}
where $I_n(z):=|z_n|^2/2$ are {\em action variables} for any $n\ge 1$ (cf. \eqref{eq:h_{r,0}}).
In addition, by \cite[Corollary 7.1]{GK}, the Birkhoff coordinate functions $\zeta_n(u)\equiv\Phi_n(u)$, $n\ge 1$, satisfy 
the (well defined) canonical relations
\begin{equation}\label{eq:canonical_relations}
\{\zeta_n,\zeta_m\}=0,\quad\{\zeta_n,\overline{\zeta}_m\}=-i\delta_{nm},\quad\forall n,m\ge 1,
\end{equation}
with respect to the Gardner bracket \eqref{eq:bracket} on $H^s_{r,0}$ for $s\ge 0$.

In the present paper we prove that for any $s>-1/2$ the Benjamin-Ono equation admits an analytic Birkhoff normal form
in an open neighborhood of zero in $H^{s}_{r,0}$ (see Theorem \ref{th:Phi} below). 
Let us first recall the notion of an analytic Birkhoff normal form in the case 
of a Hamiltonian system on $\R^{2N}$ with coordinates $\big\{(x_1,y_1,...,x_N,y_N)\big\}$ and canonical symplectic structure 
$\Omega=\sum_{n=1}^N dx_n\wedge dy_n$. Assume that the Hamiltonian $H\equiv H(x,y)$ is real analytic and
\begin{equation}\label{eq:H}
H(x,y)=\sum_{n=1}^N\omega_n\,\frac{x_n^2+y_n^2}{2}+O(r^3),\quad r:=\sqrt{|x|^2+|y|^2},
\end{equation}
where $O(r^3)$ stands for terms of order $\ge 3$ in the Taylor expansion of $H$ at zero and $\omega_n\in\R$, $1\le n\le N$.
Note that then the Hamiltonian vector field $X_H$ of $H$ has a singular point at zero and its linearization 
$d_0X_H : T_0\R^{2N}\to T_0\R^{2N}$ at zero has imaginary eigenvalues, $\spec d_0X_H =\{\pm i\omega_1,...,\pm i\omega_N\}$.
By definition, the Hamiltonian system corresponding to $X_H$ has {\em an analytic Birkhoff normal form} in an open neighborhood of zero if 
there exist open neighborhoods $U$ and $V$ of zero in $\R^{2N}$ and a real analytic canonical diffeomorphism 
\[
\mathcal{F} : U\to V,\quad(x_1,y_1,...,x_N,y_N)\mapsto(q_1,p_1,...,q_N,p_N),
\] 
such that the Hamiltonian $H\circ\mathcal{F}^{-1} : V\to\R$ has the form
\[
H\circ\mathcal{F}^{-1}(q_1,p_1,...q_N,p_N)=\mathcal{H}(I_1,...,I_N)\,\quad I_n:=\frac{p_n^2+q_n^2}{2},\quad 1\le n\le N.
\]
In this case, the quantities $I_1$,...,$I_N$ are Poisson commuting integrals and the Hamiltonian equations can be easily 
solved in terms of explicit formulas. A classical theorem of Birkhoff \cite{Birk} states that in the case when the coefficients 
$\omega_1,...,\omega_N$ in \eqref{eq:H} are rationally independent then $\mathcal{F} : U\to V$
can be constructed in terms of a {\em formal} power series. However, this formal power series generically diverges because
of small denominators. 
In the case when the Hamiltonian system is completely integrable in an open neighborhood of zero and the integrals satisfy 
a non-degeneracy condition at zero, then the Hamiltonian system admits an analytic Birkhoff normal form (Vey \cite{Vey}).
Vey's result was subsequently improved -- see \cite{Zung} and the references therein.
We remark that typically, the Hamiltonians of infinite dimensional Hamiltonian systems appearing in applications
such as the Benjamin-Ono equation, the Korteweg-de Vries equation, or the NLS equation, are only defined on
a continuously embedded proper subspace of the corresponding phase space. The following theorem can be considered as an
instance of a Hamiltonian system of infinite dimension that admits an analytic Birkhoff normal form in 
an open neighborhood of zero.


\begin{Th}\label{th:Phi}
For any $s>-1/2$ there exists an open neighborhood $U\equiv U^{s}$ of zero in $H^{s}_{c,0}$ such that
the Birkhoff map of the Benjamin-Ono equation  
\begin{equation}\label{eq:Birkhoff_map}
\Phi : H^{s}_{r,0}\to\h^{\frac{1}{2}+s}_{r,0},
\end{equation}
introduced in \cite{GK,GKT1}, extends to an analytic map $\Phi : U\to\h^{\frac{1}{2}+s}_{c,0}$.
The restriction $\Phi\big|_{U\cap H^{s}_{r,0}} : U\cap H^{s}_{r,0}\to\h^{\frac{1}{2}+s}_{r,0}$ is
a real analytic diffeomorphism onto its image. 
In particular, in view of \eqref{eq:BO-Hamiltonian_in_BNF} and \eqref{eq:canonical_relations}, 
for $s>-1/2$ the Benjamin-Ono equation \eqref{eq:BO} has an analytic Birkhoff normal form in the
open neighborhood $U\cap H^{s}_{r,0}$ of zero in $H^{s}_{r,0}$ such that the Hamiltonian $H\circ\Phi^{-1}$ 
given by \eqref{eq:BO-Hamiltonian_in_BNF} is well defined and real analytic on $\h^1_{r,0}$.
\end{Th}

\begin{Rem}
Since, prior to this work, the Benjamin-Ono equation was not known to be integrable in terms of
an appropriate infinite dimensional real analytic setup, Theorem \ref{th:Phi} should {\em not} be viewed 
as an instance of Vey's result for the infinite dimensional Hamiltonian system \eqref{eq:BO}.
For the same reason, the techniques developed by Kuksin $\&$Perelman in \cite{KukPer} do not apply in this case.
\end{Rem}

We apply Theorem \ref{th:Phi} to further analyze regularity properties of the solution map
of the Benjamin-Ono equation. Assume that $s>-1/2$.
For $u_0 \in H^{s}_{r, 0}$, denote by  $t \mapsto u(t) \equiv u(t, u_0)$ the solution of the Benjamin-Ono equation \eqref{eq:BO} 
with initial data $u_0$, constructed in \cite{GKT1}. 
For given $t\in\R$ and $T>0$ consider the {\em flow map}
\[
\mathcal{S}_0^t : H^{s}_{r, 0}\to H^{s}_{r, 0},\quad u_0\mapsto u(t, u_0),
\]
as well as the {\em solution map}
\[
\mathcal{S}_{0,T} : H^{s}_{r, 0}\to C\big([-T, T], H^{s}_{r, 0}\big),\quad u_0\mapsto u|_{[-T,T]},
\]
of the Benjamin-Ono equation. To state our results, we need to introduce one more definition.  
A continuous map $F : X\to Y$ between two Banach spaces $X$ and $Y$ is called 
{\em nowhere locally uniformly continuous} in an open neighborhood $U$ in $X$ if
the restriction $F\big|_V : V\to Y$ of $F$ to any open neighborhood $V\subseteq U$ is {\em not} uniformly continuous. 
In a similar way one defines the notion of a {\em nowhere locally Lipschitz} map in an open neighborhood $U\subseteq X$.

\begin{Th}\label{th:well-posedness}
\begin{itemize}
\item[(i)] For any $-1/2<s<0$ and $t\ne 0$, the flow map $\mathcal{S}_0^t : H^{s}_{r,0}\to H^{s}_{r,0}$
of the Benjamin-Ono equation \eqref{eq:BO} is nowhere locally uniformly continuous in an open neighborhood $U$
of zero in $H^{s}_{r,0}$. In particular, $\mathcal{S}_0^t : H^{s}_{r,0}\to H^{s}_{r,0}$ is nowhere locally
Lipschitz in $U$.
\item[(ii)] For any $s\ge 0$ there exists an open neighborhood $U\equiv U^s$ of zero in $H^s_{c,0}$ so that
for any $T>0$, the solution map $\mathcal{S}_{0,T} : U\cap H^s_{r, 0}\to C\big([-T, T], H^s_{r, 0}\big)$ is real analytic.
\end{itemize}
\end{Th}

\noindent{\bf Addendum to Theorem \ref{th:well-posedness}(ii).}\label{rem:improved_well-posedness}
{\em 
For any $k\ge 1$, $s>-1/2+2k$, and $T > 0$, the solution map
\[
\mathcal{S}_{0,T} : H^s_{r, 0}\to\bigcap_{j=0}^k C^j\big([-T, T], H^{s-2j}_{r, 0}\big)
\]
is well defined and real analytic in an open neighborhood of zero in $H^{s}_{r,0}$. 
}

\begin{Rem}
Item (i) of Theorem \ref{th:well-posedness} improves on the  result by Molinet in \cite[Theorem 1.2]{Mo},
saying that for any $s<0$, $t\in\R$, the flow map $\mathcal{S}^t_0$ (if it exists at all) is not of class 
$C^{1, \alpha}$ for any $\alpha>0$.
Item (ii)  of Theorem \ref{th:well-posedness} improves on the result by Molinet, saying that for any $s\ge 0$, $t\in\R$, 
the flow map $\mathcal S^t_0$ is real analytic near zero (\cite[Theorem 1.2]{Mo}).
\end{Rem}

\begin{Rem}
Any solution $u$ of the Benjamin-Ono equation in $H^s_{r, 0}$, $s> -1/2$, constructed in \cite{GKT1}, has the property 
that for any $c\in\R$, $u(t, x - 2 ct) + c$ is again a solution with constant mean value $c$. 
It is straightforward to see that for any $c\in\mathbb R$, Theorem \ref{th:well-posedness} holds on the affine space 
$\big\{u\in H^s_r \,\big| \hu(0)=c\big\}$. 
\end{Rem}

\begin{Rem}
Note that in Theorem \ref{th:well-posedness} we restrict our attention to solutions with mean value zero. 
In case when the mean value of solutions is not prescribed, the flow map is no longer locally uniformly continuous 
on $H^s_r$ with $s\ge 0$ as can be seen by considering families of solutions $c+u(t,x-2ct)$ of the Benjamin-Ono equation, 
parametrized by $c\in \R$, where $c$ tends to zero and $u$ is oscillating. 
\footnote{This fact can also be proved in a very transparent way using Birkhoff coordinates (cf. \cite[Appendix A]{KT1}).}
Based on this observation, Koch $\&$Tzvetkov succeeded in proving that the solution map of 
the Benjamin-Ono equation on the line is not uniformly continuous on bounded subsets of initial data in 
$H^s(\R,\R)$ with $s>0$ (cf. \cite{KoTz}).
We point out that Theorem \ref{th:well-posedness}(i) does {\em not} follow from the lack of uniform continuity on 
bounded subsets of initial data of the solution map of the Benjamin-Ono equation on $H^s_r$, $s > -1/2$, 
and its proof requires new arguments.
We remark that the Birkhoff coordinates allow to explain in clear terms why on any neighborhood of
zero in $H^s_{r,0}$ with $-1/2 < s < 0$, the solution map of the Benjamin-Ono equation is nowhere 
locally uniformly continuous: the reason is that on these neighborhoods, the frequencies $\omega_n$ 
of the Benjamin-Ono equation are {\em not} locally Lipschitz continuous uniformly in $n\ge 1$.
It follows from the analyticity of the Birkhoff map on all of $H^s_{r,0}$, established in \cite{GKT3}, 
that these properties of the solution map of the Benjamin-Ono equation actually hold on the 
corresponding entire Sobolev spaces. 
\end{Rem}

\medskip

\noindent{\em Ideas of the proofs.} 
The Birkhoff map $\Phi : H^{s}_{r,0}\to\h^{\frac{1}{2}+s}_{r,0}$, $u \mapsto (\Phi_n(u))_{|n|\ge 1}$,
is defined in terms of the Lax operator $L_u = -i\partial_x - T_u$ where $T_u$ denotes
the Toeplitz operator with potential $u\in  H^{s}_{r,0}$, $s > -1/2$. 
We refer to Section \ref{Lax operator} for a review of  terminology and results concerning this operator,
established in the previous papers \cite{GK}, \cite{GKT1}, and \cite{GKT2}.
At this point we only mention that the spectrum of $L_u$ is discrete and consists of a sequence of simple, real eigenvalues,
bounded from below, $\lambda_0(u) < \lambda_1(u) < \cdots$. Furthermore, there exist $L^2-$normalized 
eigenfunctions $f_n(u)$ of $L_u$, corresponding to the eigenvalues $\lambda_n(u)$, which are uniquely determined
by the normalization conditions
\begin{equation}\label{norm'zation f_n}
\1 f_0(u) | 1  \2 > 0, \qquad \1 e^{ix} f_{n-1}(u) | f_n(u) \2 > 0, \quad \forall n \ge 1 .
\end{equation}
The components of the Birkhoff map $\Phi$ are then defined as 
\begin{equation}
\Phi_n(u) =\frac{ \1 1| f_n(u)  \2}{ \sqrt[+]{\kappa_n(u)}} , \qquad \forall n \ge 1,
\end{equation}
where $\kappa_n(u) > 0$ are scaling factors (cf. Section \ref{sec:the_Birkhof_map}).
We point out that the normalization conditions  \eqref{norm'zation f_n} are defined inductively.
It is this fact which makes it more difficult to prove that $\Phi$ extends to an analytic map.

For $u = 0$, one has
\[
\lambda_n(0) = n , \quad f_n(0) = e^{inx} , \quad \kappa_n(0) = 1 , \qquad  \forall n \ge 0 .
\]
We then use perturbation arguments to show that for any $s > -1/2$, $\Phi$ extends to an analytic map on 
a neighborhood of zero in $H^s_{c,0}$. Let us outline the main steps  of the proof of this result.

In \cite{GKT2} we proved that there exists a neighborhood $U\equiv U^s$ of zero in $H^s_{c,0}$
so that for any $u\in U$, the spectrum of $L_u$ consists of simple eigenvalues $\lambda_n(u)$, $n \ge 0$.
These eigenvalues are analytic maps on $U$ and so are the Riesz projectors $P_n(u)$
onto the one dimensional eigenspaces, corresponding to these eigenvalues.
See Proposition \ref{prop:L-global} and Proposition \ref{prop:L-local} in Section \ref{Lax operator} for a review of these results.
It then follows that for any $n \ge 0$, $h_n(u) = P_n(u) e^{inx}$  is an analytic function $U \to H^{1+s}_+$ 
(see \eqref{eq:H_+} below).

In Section \ref{sec:Psi} we introduce the pre-Birkhoff map 
$$
\Psi(u) = (\Psi_n(u))_{n \ge 1}, \qquad 
\Psi_n(u) = \1 h_n(u) | 1 \2 , \quad \forall n \ge 1.
$$
With the help of the Taylor expansion of $P_n(u)$, $n \ge 1$, we show that for any $u \in U$, $\Psi(u) \in\h^{1+s}_+$
and that $\Psi : U \to\h^{1+s}_+$ is analytic (see \eqref{def frak h_+} below).  The arguments developed allow to prove that
these results hold for any $s > -1/2$, not just for $-1/2 < s \le 0$.

In Section \ref{sec:the_Birkhof_map}, we relate the Birkhoff map $\Phi$ to the pre-Birkhoff map $\Psi$
by using in addition to $\kappa_n(u)$, $n \ge 0$, the scaling factors $\mu_n(u)$, $n \ge 1$, introduced in \cite{GK} and further analyzed 
in \cite{GKT2}. The main ingredient for proving that $\Phi$ extends to an analytic map on a neighborhood of zero in $H^s_{c, 0}$
is a novel {\em Vanishing Lemma} (cf. Lemma \ref{lem:D=0}), which is discussed and proved in Section  \ref{sec:the_delta_map}.

Theorem  \ref{th:Phi} and Theorem \ref{th:well-posedness} are proved in Section \ref{proof main results}.
To prove Theorem \ref{th:well-posedness} we use arguments developed in \cite{KM1} 
to show a corresponding result for the Korteweg-de Vries equation.

\medskip

\noindent{\em Related work.} 
In \cite{GKT3} we prove an extension of Theorem \ref{th:Phi}, saying that for any $s > -1/2$, 
the Birkhoff map $\Phi : H^{s}_{r,0}\to\h^{\frac{1}{2}+s}_{r,0}$ is a real analytic diffeomorphism.
The proof of this result uses, in broad terms, the same strategy developed to prove Theorem \ref{th:Phi}, but 
it is more technical, relying  on the approximation of arbitrary elements in $H^{s}_{r,0}$ by finite gap potentials
and on properties of the spectrum of the Lax operator, associated to such potentials.
As in the proof of Theorem \ref{th:Phi}, one of the main ingredients of the proof of the analytic extension is
Lemma \ref{lem:D=0} (Vanishing Lemma) of Section \ref{sec:the_delta_map}.

The result saying that  the Birkhoff map $\Phi : H^{s}_{r,0}\to\h^{\frac{1}{2}+s}_{r,0}$ is a real analytic diffeomorphism
for the appropriate range of $s$, shows that similarly as for the Korteweg-de Vries (KdV) equation (cf. \cite{KP-book}, \cite{KT1}) 
and the defocusing nonlinear Schr\"odinger (NLS) equation (cf. \cite{GK-book}), the Benjamin-Ono equation on the torus is 
integrable in the strongest possible sense. We point out that the proof of the analyticity of the Birkhoff map in the case of 
the Benjamin-Ono equation significantly differs from the one in the case of the KdV and NLS equations due to the fact 
the Benjamin-Ono equation is not a differential equation.

The above result on the Birkhoff map $\Phi$ can be applied to prove in a straightforward way that Theorem \ref{th:well-posedness} 
extends to all of $H^{s}_{r,0}$. We remark that a result of this type has been first proved for the KdV equation in 
\cite[Theorem 3.10]{KM1}. It turns out that the analysis of the solution map of the KdV equation, expressed in Birkhoff coordinates, 
can also be used to prove Theorem  \ref{th:well-posedness} and its extension.

Similarly as in the case of the KdV equation (cf. \cite{KP-book}, \cite{Kuk}) and the NLS equation (cf. \cite{BKM} and references therein),  
a major application of the result, saying that $\Phi : H^{s}_{r,0}\to\h^{\frac{1}{2}+s}_{r,0}$ is a real analytic diffeomorphism for 
any $s>-1/2$, concerns its use to study (Hamiltonian) perturbations of the Benjamin-Ono equation by KAM methods near finite gap 
solutions of arbitrary large amplitude. 

\medskip

\noindent{\em Notation.} In this paragraph we summarize the most frequently used notations in the paper. 
For any $\beta\in\R$, $H^\beta_c$ denotes the Sobolev space $H^\beta\big(\T,\C\big)$
of complex valued functions on the torus $\T = \R/2\pi\Z$ with regularity exponent $\beta$. 
The norm in $H^\beta_c$ is given by 
\[
\|u\|_\beta:=\Big(\sum_{n\in\Z}\1 n\2^{2\beta}|\hu(n)|^2\Big)^{1/2}
\]
where $\hu(n)$, $n\in\Z$, are the Fourier coefficients of $u\in H^\beta_c$.
For $\beta=0$ and $u\in H^0_c\equiv L^2(\T,\C)$ we set $\|u\|\equiv \|u\|_0$.
By $H^\beta_{c,0}$ we denote the complex subspace in $H^\beta_c$ of functions with mean value zero,
\[
H^\beta_{c,0}=\big\{u\in H^\beta_c\,\big|\,\hu(0)=0 \big\}.
\]
For $f\in H^\beta_c$ and $g\in H^{-\beta}_c$, define the sesquilinear and bilinear pairing,
\[
\1 f|g\2:=\sum_{n\in\Z}\widehat{f}(n)\overline{\widehat{g}(n)} ,
\qquad \quad
\1 f , g\2:=\sum_{n\in\Z}\widehat{f}(n)\widehat{g}(-n).
\]
The positive Hardy space $H^\beta_+$ with regularity exponent $\beta\in\R$ 
is defined as
\begin{equation}\label{eq:H_+}
H^{\beta}_+:=\big\{f\in H^{s}_c\,\big|\,\widehat{f}(n)=0\,\,\,\forall n<0\big\} .
\end{equation}
For $1\le p<\infty$ denote by
$
\ell^p_+\equiv\ell^p\big(\Z_{\ge 1},\C\big)
$
the Banach space of complex valued sequences $ z=(z_n)_{n\ge 1} $ with finite norm
\[
\| z \|_{\ell^p_+} :=\Big(\sum_{n\ge 1}|z_n|^p\Big)^{1/p}<\infty.
\]
Similarly, denote by $\ell^\infty_+\equiv\ell^\infty\big(\Z_{\ge 1},\C\big)$ the Banach space of  
complex valued sequences with finite supremum  norm $ \| z\|_{\ell^\infty_+} :=\sup_{n\ge 1}|z_n|$. 
More generally, for $1\le p\le\infty$ and $m\in\Z$, we introduce the Banach space
\[
\ell^p_{\ge m}\equiv\ell^p\big(\Z_{\ge m},\C\big),\qquad \Z_{\ge m}:=\{n\in\Z\,| \, n\ge m\},
\]
as well as the space $\ell^p_c\equiv\ell^p\big(\Z,\C\big)$, defined in a similar way.
Furthermore, denote by $\h^\beta_+$, $\beta\in\R$, the Hilbert space 
of complex valued sequences $ z = (z_n)_{n \ge 1}$ with 
\begin{equation}\label{def frak h_+}
\|  z\|_{\h^\beta_+} :=\Big(\sum_{n\ge 1}|n|^{2\beta}|z_n|^2\Big)^{1/2}<\infty.
\end{equation}
In a similar way, we define the space of complex valued sequences $\h^\beta_{\ge m}$
for any $m\in\Z$ and $\beta \in \R$.  
Finally, for $z\in\C\setminus(-\infty,0]$, 
we denote by $\sqrt[+]{z}$ the principal branch of the  square root of $z$, defined by
$\re\big(\sqrt[+]{z}\big)>0$.

\section{The Lax Operator}\label{Lax operator}
In this section we review results from \cite{GKT2} on the Lax operator $L_u$ of the Benjamin-Ono equation 
with potentials $u$ in $H^{s}_c$ with $s > -1/2$, used in this paper. 
Throughout this section we assume that  $-1/2 < s \le 0$.

\smallskip

For $u\in H^{s}_c$ with $-1/2 < s \le 0,$ consider the pseudo-differential expression
\begin{equation}\label{eq:L}
L_u:=D-T_u
\end{equation}
where $D = -i\partial_x$,  $T_u: H^{1+s}_+\to H^{s}_+$ is the Toeplitz operator with potential $u$, 
\begin{equation*}
T_u f:=\Pi(u f),\quad f\in H^{1+s}_+,
\end{equation*}
and $\Pi\equiv\Pi^+ : H^{s}_c \to H^{s}_+$ is the Szeg\H o projector
\[
\Pi : H^{s}_c\to H^{s}_+,\quad \sum_{n\in\Z}\widehat v(n) e^{i n x}\mapsto\sum_{n\ge 0}\widehat v(n) e^{i n x},
\]
onto the (positive) Hardy space $H^{s}_+$, introduced in \eqref{eq:H_+}.
Note that when restricted to $H^{1+s}_+$, $D$ coincides with the Fourier multiplier $|\partial_x|$,
\[
|\partial_x| : H^{1+s}_c\to H^{s}_c,\quad\sum_{n\in\Z}\hat{f}(n)\mapsto\sum_{n\in\Z}|n|\hat{f}(n),
\]
It follows from Lemma 1 in \cite{GKT2} that $L_u$ defines an operator in $H^{s}_+$ with domain
$H^{1+s}_+$ so that the map $L_u : H^{1+s}_+\to H^{s}_+$ is bounded.
The following result follows from \cite[Theorem 1]{GKT2} and \cite{GKT1}.

\begin{Prop}\label{prop:L-global}
For any $-1/2 < s \le 0$, there exists an open neighborhood $W\equiv W^{s}$ of $H^{s}_{r,0}$ in $H^{s}_{c,0}$ so that 
for any $u\in W$ the operator $L_u$ is a closed operator in $H^{s}_+$ with domain $H^{1+s}_+$. 
The operator has a compact resolvent and all its eigenvalues are  {\em simple}. When appropriately listed,  
$\lambda_n \equiv \lambda_n(u)$, $n\ge 0$, 
satisfy $\re(\lambda_n)<\re(\lambda_{n+1})$ for any $n \ge 0$ and $|\lambda_n-n|\to 0$ as $n\to\infty$. 
For $u\in H^{s}_{r,0}$ the eigenvalues are real valued and 
$\gamma_n(u):= \lambda_n(u)-\lambda_{n-1}(u)-1\ge 0$ for $n\ge 1$.
\end{Prop}

For $u$ in an open neighborhood of zero in $H^{s}_{c,0}$, $-1/2 < s \le 0$, the proposition above can be specified as follows
(see \cite[Corollary 2, Proposition 1]{GKT2}). For $\varrho>0$ consider the sets in $\C$,
\begin{equation}\label{eq:Vert}
\Wert_n(\varrho):=\big\{\lambda\in\C\,\big|\,|\lambda-n|\ge\varrho,\,|\re(\lambda)-n|\le 1/2\big\},\quad n\ge 0,
\end{equation}
\begin{equation}\label{def:D_n}
D_n(\varrho):=\big\{\lambda\in\C\,\big|\,|\lambda-n|<\varrho\big\},\quad n\ge 0,
\end{equation}
and let $\partial D_n(\varrho)$ be the counterclockwise oriented boundary of $D_n(\varrho)$ in $\C$.
Denote by $\Wert_n^0(\varrho)$ the interior of $\Wert_n(\varrho)$ 
in $\C$.

\begin{Prop}\label{prop:L-local}
For any $-1/2 < s \le 0$, there exists an open neighborhood $U\equiv U^{s}$ of zero in $H^{s}_{c,0}$ so that 
for any $u\in U$ the operator $L_u$ is a closed operator in $H^{s}_+$ with domain $H^{1+s}_+$. 
The operator has a compact resolvent and all its eigenvalues are simple. When listed appropriately, they satisfy
\[
\lambda_n\in D_n(1/4),\quad n\ge 0.
\]
Moreover, for any $n\ge 0$ the resolvent map 
\[
U\times\Wert_n^0(1/4)\to\LL\big(H^{s}_+,H^{1+s}_+\big),\quad(u,\lambda)\mapsto(L_u-\lambda)^{-1},
\]
is analytic. In addition, $(L_u-\lambda)^{-1}=(D-\lambda)^{-1}\big(I-T_u(D-\lambda)^{-1}\big)^{-1}$
and for any $n\ge 0$, the Neumann series 
\begin{equation}\label{eq:neumann_series}
\big(I-T_u(D-\lambda)^{-1}\big)^{-1}=\sum_{m\ge 0}[T_u(D-\lambda)^{-1}]^m
\end{equation}
converges in $\LL\big(H^{s}_+\big)$ absolutely and uniformly for $(u,\lambda)\in U\times\Wert_n^0(1/4)$.
\end{Prop}

\begin{Rem}\label{rem:shifted_norm}
By \cite[Corollary 2]{GKT2} the neighborhood $U$ in Proposition \ref{prop:L-local} can 
be chosen so that for any $n\ge 0$,
$$
\big\|T_u(D-\lambda)^{-1}\big\|_{\LL(H^{s;n}_+)}<1/2, \qquad \forall  (u,\lambda)\in U\times\Wert_n(1/4) ,
$$
 where $H^{s;n}_+$ stands for the space $H^{s}_+$ equipped with the equivalent
``shifted'' norm $\|f\|_{s;n}:=\big(\sum_{k\ge 0}\1n-k\2^{2s}|\widehat{f}(k)|^2\big)^{1/2}$
(cf. \cite[Lemma 3]{GKT2}). In particular, for any $n\ge 0$ the series in \eqref{eq:neumann_series} converges 
uniformly on $U\times\Wert_n(1/4)$ with respect to an operator norm which is equivalent to the operator norm in 
$\LL\big(H^{s}_+\big)$. 
\end{Rem}

\begin{Rem}\label{rem:finite-gap}
For any $u\in H^{s}_{r,0}$, $-1/2 < s \le 0$, an analogue of Proposition \ref{prop:L-local} 
holds in an open neighborhood $U_u$ of $u$ in $H^{s}_{c}$
 -- see \cite[Section 5]{GKT2} for details.
\end{Rem}

\noindent Given any $ -1/2 < s \le 0$, Proposition \ref{prop:L-local} implies that  for any $u \in U$ and $ n\ge 0$, the Riesz projector
\begin{equation}\label{eq:P_n-local}
P_n(u):=
-\frac{1}{2\pi i}\oint_{\partial D_n}(L_u-\lambda)^{-1}\,d\lambda\in\LL\big(H^{s}_+,H^{1+s}_+\big),
\quad D_n := D_n(1/3), 
\end{equation}
is well defined and that the map
\begin{equation}\label{P_n anal near 0}
P_n : U\to\LL\big(H^{s}_+,H^{1+s}_+\big),\quad u\mapsto P_n(u).
\end{equation}
is analytic.


\section{Analytic extension of the pre-Birkhoff map}\label{sec:Psi}
In this section we introduce the pre-Birkhoff map and study its properties. 
Throughout this section, we assume that $s > -1/2$ and define
\begin{equation}\label{def sigma}
\sigma\equiv \sigma(s) := \min (s, 0).
\end{equation}

It follows from Proposition \ref{prop:L-local} (cf. \eqref{P_n anal near 0})
that for any $s>-1/2$, there exists 
an open neighborhood $U^\sigma$ of zero in $H^{\sigma}_{c,0}$ so that for any $u\in U^\sigma$ and $n\ge 0$
the Riesz projector
\begin{equation*}
P_n(u)=
-\frac{1}{2\pi i}\oint\limits_{\partial D_n}(L_u-\lambda)^{-1}\,d\lambda\in\LL\big(H^{\sigma}_+,H^{1+\sigma}_+\big),
\end{equation*}
is well defined and the map $U\to\LL\big(H^{\sigma}_+,H^{1+\sigma}_+\big)$, $u\mapsto P_n(u)$, is analytic. 
Here $\partial D_n$ is the counterclockwise oriented boundary of $D_n\equiv D_n(1/3)$ (see \eqref{def:D_n}).
Hence, for any $n\ge 0$ the map
\begin{equation}\label{eq:Psi_n-map}
U^\sigma \to H^{1+\sigma}_+,\quad u\mapsto h_n( u),
\end{equation}
is analytic, where
\begin{equation}\label{eq:h_n}
h_n( u):=P_n(u) e_n\in H^{1+\sigma}_+ ,\qquad e_n := e^{i n x} .
\end{equation}
The main objective in this section is to prove the following 

\begin{Prop}\label{prop:Psi}
For any $s>-1/2$ there exists an open neighborhood $ U^{s}$ of zero in $H^{s}_{c,0}$ so that
the following holds:
\begin{itemize}
\item[(i)] The map,
\begin{equation}\label{eq:Psi}
\Psi : U^s \to\h^{1+s}_+,\quad u\mapsto\big(\1 1,h_n(u)\2\big)_{n\ge 1},
\end{equation}
is analytic. We refer to $\Psi$ as the pre-Birkhoff map (near zero).
\item[(ii)] For any $n\ge 0$, the map $U^s \to\C$, $u\mapsto\1 h_n( u)|e_n\2$,
is analytic and 
\[
\big|\1 h_n( u)|e_n\2\big|\ge 1/2, \qquad \forall \, u \in U^s .
\] 
\end{itemize}
\end{Prop}

\begin{proof}[Proof of Proposition \ref{prop:Psi}]
Take $s>-1/2$ and let $ U^{\sigma}$ be an open neighborhood of zero in $H^{\sigma}_{c,0}$ so that
the statement of Proposition \ref{prop:L-local} holds. Then, by the discussion above,
the map \eqref{eq:Psi_n-map} is analytic for any $n\ge 0$. In particular,  the components
$\big\1 1,h_n(u)\big\2$, $n \ge 1$, of the map \eqref{eq:Psi} are analytic on $U^\sigma$. 
Hence, item {\rm(i)} will follow from \cite[Theorem A.5]{KP-book} once we prove
that the restriction of the map \eqref{eq:Psi} to a sufficiently small neighborhood $U^s \subseteq U^\sigma \cap H^s_{c,0}$ of 
zero in $H^s_{c,0}$ is bounded. Let us show this.
It follows from Proposition \ref{prop:L-local} that for any $n\ge 1$ and 
for any $u\in U^\sigma$ we have 
\begin{align}
\Psi_n(u) := \big\1 1,h_n( u)\big\2&=-\frac{1}{2\pi i}\oint_{\partial D_n}
\big\1(L_u-\lambda)^{-1}e_n,1\big\2\,d\lambda\nonumber\\
&=\frac{1}{2\pi i}\sum\limits_{m\ge 1}\oint_{\partial D_n}\big\1[T_u(D-\lambda)^{-1}]^me_n,1\big\2
\,\frac{d\lambda}{\lambda}\nonumber\\
&=\Psi^{(1)}_n(u)+\sum\limits_{m\ge 1}\Psi^{(m+1)}_n(u)\label{eq:Psi-expansion}
\end{align}
where
\begin{equation}\label{eq:Psi-linear_term}
\Psi^{(1)}_n(u):=-\frac{\hu(-n)}{n} , 
\end{equation}
\begin{equation}\label{eq:Psi-higher_order_terms}
\Psi^{(m+1)}_n(u):=\sum_{\tb{k_j\ge 0}{1\le j\le m}}  \Psi_{n, u}(k_1,...,k_m),\qquad m\ge 1,
\end{equation}
and
\begin{equation}\label{eq:Psi_u}
\Psi_{n, u}(k_1,...,k_m) :=\frac{1}{2\pi i}\oint_{\partial D_n}
\frac{\hu(-k_m)}{k_m-\lambda}\frac{\hu(k_m-k_{m-1})}{k_{m-1}-\lambda}\cdots\frac{\hu(k_1-n)}{n-\lambda}
\,\frac{d\lambda}{\lambda}.
\end{equation}
By Proposition \ref{prop:L-local} and Remark \ref{rem:shifted_norm}, the series in \eqref{eq:Psi-expansion} and 
\eqref{eq:Psi-higher_order_terms} converge absolutely and uniformly on $U^\sigma$. Moreover, by Remark \ref{rem:shifted_norm}, 
for any $n\ge 1$ and for any $m\ge 0$, $\Psi^{(m+1)}_n : H^{\sigma}_+\to\C$ is a bounded polynomial map of order $m+1$.
Hence, for any $n\ge 1$ the series \eqref{eq:Psi-expansion} is the Taylor's expansion of $\Psi_n : U^\sigma \to\C$ at zero $u=0$. 
For $m=1$, one obtains from \eqref{eq:Psi-higher_order_terms}, \eqref{eq:Psi_u}, and  Cauchy's formula that for any $n\ge 1$,
\[
\Psi^{(2)}_n(u)=-\frac{1}{n}\sum_{k_1\ge 0,k_1\ne n}\!\!\!\hu(-k_1)\,\frac{\hu(k_1-n)}{k_1-n}
=-\frac{1}{n}\sum_{l_1\ge -n,l_1\ne 0}\!\!\!\hu(-n-l_1)\,\frac{\hu(l_1)}{l_1}.
\]
Let us now consider the general case $m\ge 1$. 
By passing to the variable $\mu:=\lambda-n$ in the contour integral \eqref{eq:Psi_u} and then setting
$l_j:=k_j-n$, $1\le j\le m$, we obtain from \eqref{eq:Psi-higher_order_terms} and \eqref{eq:Psi_u} that
\begin{equation}\label{eq:Psi-higher_order_terms'}
\Psi^{(m+1)}_n(u)=\sum_{\tb{l_j\ge-n}{1\le j\le m}} H_{u, n} (l_1,...,l_m),\quad m\ge 2,
\end{equation}
where $ H_{ u, n}(l_1, \ldots, l_m)$ is given by
\begin{equation}\label{eq:H_u}
- \frac{1}{2\pi i}\!\!\oint\limits_{\partial D_0}\!\!\frac{1}{n+\mu}
\frac{\hu(-n-l_m)}{l_m-\mu}\frac{\hu(l_m-l_{m-1})}{l_{m-1}-\mu}\cdots\frac{\hu(l_2-l_1)}{l_1-\mu}\,\hu(l_1)
\,\frac{d\mu}{\mu}.
\end{equation}
Note that the latter integral and hence $H_{ u, n}$  is defined for any $l_1, \ldots, l_m$ in $\Z$. 

To estimate the term $H_{ u, n}(l_1, \ldots, l_m)$, note that for any $\mu \in \partial D_0$, $n \ge 1$, and $\ell_1, \ldots, \ell_m \in \Z$
\begin{equation}\label{estimate l_j - mu}
| n + \mu| \ge \frac{n}{2}, \qquad \quad  |\ell_j - \mu | \ge \frac{|\ell_j| +1}{5}, \quad 1\le j \le m\, ,
\end{equation}
implying that
$$
|H_{ u, n}(l_1, \ldots, l_m)| \le 
 \frac{2}{n} 5^m \frac{|\hu(-n-l_m)|}{|l_m| +1}\frac{|\hu(l_m-l_{m-1})|}{|l_{m-1}| + 1}\cdots\frac{|\hu(l_2-l_1)|}{|l_1| + 1}\,|\hu(l_1)|
$$
By Lemma \ref{new version map Q} below, it then follows that for a y $u \in U^\sigma \cap H^s_{c, 0}$, 
\begin{align}\label{estimate H_{u,n}}
&  \sum_{n\ge 1}|n|^{2(1+s)}\Big(\sum_{\tb{l_j\ge-n}{1\le j\le m}}\big|H_{u, n}(l_1,...,l_m)\big|\Big)^2  \nonumber\\
& \le (5^{m+1})^2 \sum_{n \ge 1} |n|^{2s} 
\Big( \sum_{\tb{l_j \in \Z}{1\le j\le m}} \frac{|\hu(-n-l_m)|}{|l_m| +1} \cdots \frac{|\hu(l_2-l_1)|}{|l_1| + 1}\,|\hu(l_1)|
\Big)^2  \nonumber\\
& \le \Big( 5^{m+1}C_s^{m}\|u\|_{s}^{m+1}\Big)^2 \le  \Big( 5 C_s \|u\|_{s} \Big)^{2(m + 1)} \nonumber
\end{align}
where $C_s \ge 1$ is the constant given by Lemma \ref{new version map Q}. 
Hence, by shrinking $U^\sigma \cap H^s_{c, 0}$ to a (bounded) neighborhood $U^s \subseteq H^{s}_{c,0}$, 
the latter estimate implies that
\[
\big\|\Psi(u)\big\|_{\h^{1+s}_+}\le\sum_{m\ge 0}\big\|\Psi^{(m+1)}(u)\big\|_{\h^{1+s}_+}\le 1, \qquad  \forall \, u\in U^s.
\]
This proves item {\rm (i)}. The proof of item {\rm (ii)} is similar and hence omitted.
\end{proof}

The following estimate follows directly from \cite[Lemma 1]{GKT2}.
\begin{Lem}\label{new version map Q}
For any $ z = (z_n)_{n \in \Z} \in \h^{s}_c$
with $s > -1/2$, the linear operator
\[
Q(z) : \h^{s}_c \to \h^{s}_c , \,  y = (y_k)_{k \in \Z} \mapsto 
Q(z)[ y] :=\Big( \sum_{k \in \Z} \frac{z_{\ell - k}}{|k| + 1} y_k \Big)_{\ell \in \Z},
\]
is well defined and there exists a constant $C_s \ge 1$ so that
\[
\| Q(z)[ y]\|_s \le C_s \| z\|_s  \| y\|_s , \qquad \forall \, z,  y \in \h^{s}_c.
\]
\end{Lem}

\bigskip


\section{Analytic extension of the Birkhoff map in a neighborhood of zero}\label{sec:the_Birkhof_map}

The goal of this section is to extend for any $s > -1/2$ the Birkhoff map, defined on $H^{s}_{r,0}$ \big(cf. \cite{GK} ($s=0$),
\cite{GKT1} ($s > -1/2$)\big), to an analytic map on an open neighborhood of zero in $H^{s}_{c,0}$. 

Recall from \cite{GKT1} that for any $ s > -1/2 $, the Birkhoff map is given by
\begin{equation}\label{eq:Phi_n}
\Phi : H^{s}_{r,0}\to\h^{\frac{1}{2}+s}_+,\,
u\mapsto \big(\Phi_n(u)\big)_{n\ge 1} ,
\qquad
\Phi_n(u)\equiv\zeta_n(u):=\frac{\big\1 1|f_n( u)\big\2}{\sqrt[+]{\kappa_n(u)}},
\end{equation}
where the sequence space $\h^{\frac{1}{2}+s}_+$ is defined in \eqref{def frak h_+},
the eigenfunctions $f_n\equiv f_n(u)$, $n \ge 0$, of the Lax operator $L_u$, corresponding to the
eigenvalues $\lambda_n \equiv \lambda_n(u)$, are normalized so that $\| f_n \| = 1$ and (cf. \cite[Definition 2.1]{GK}, \
\cite[Lemma 6]{GKT1}),
\begin{equation}\label{normalisation f_n}
\langle 1 | f_0 \rangle > 0, \qquad  \langle e^{ix}f_{n-1} | f_n \rangle > 0, \ \  \forall \, n \ge 1,
\end{equation}
and the norming constants $\kappa_n \equiv \kappa_n(u) > 0$, $n \ge 0$, are given by (cf. 
\cite[Corollary 3.1]{GK}, \cite[(29)]{GKT1})
\begin{equation}\label{eq:kappa_n}
\kappa_0:=\prod_{k\ge 1}\Big(1-\frac{\gamma_k}{\lambda_k-\lambda_0}\Big) , \quad
\kappa_n:=\frac{1}{\lambda_n-\lambda_0}
\prod_{n\ne k\ge 1}\Big(1-\frac{\gamma_k}{\lambda_k-\lambda_n}\Big)  ,\quad n\ge 1,
\end{equation}
with
$\gamma_k\equiv\gamma_k(u)$ defined as in Proposition \ref{prop:L-global}.
To construct the analytic extension of $\Phi$, we express the normalisation conditions \eqref{normalisation f_n} 
of the eigenfunctions $(f_n)_{n \ge 0}$ in terms of $(\kappa_n)_{n \ge 0}$, and the norming constants 
$\mu_n\equiv \mu_n(u) > 0$, $n \ge 1$, defined for $u \in H^s_{r,0}$, $s >  -1/2 $, by  \big(cf. \cite[Remark 4.1]{GK}, 
\cite[formula (28)]{GKT2} ($s=0$),   \cite[Section 3]{GKT1}($ -1/2< s < 0$)\big),
\begin{equation}\label{eq:mu_n}
\mu_n:=\Big(1-\frac{\gamma_n}{\lambda_n-\lambda_0}\Big)\prod_{n\ne k\ge 1}
\Big(1-\gamma_n\frac{\gamma_k}{(\lambda_{k-1}-\lambda_{n-1})(\lambda_k-\lambda_n)}\Big) .
\end{equation}
For any $u \in H^s_{r,0}$ and $n \ge 0$, the projection $P_n \equiv P_n(u)$, introduced in \eqref{eq:P_n-local},
coincides with the orthogonal projection
$$
 H^s_+ \to  H^{s + 1/2}_+,  \, f \mapsto \langle f | f_n\rangle f_n .
$$
Expressed in terms of $P_n$, 
the normalisation conditions \eqref{normalisation f_n} of $f_n$ read
\big(cf.  \cite[Remark 4.1, Corollary 3.1]{GK} ($s = 0$), \cite[Lemma 6]{GKT1} ($ -1/2< s \le 0$)\big)
\begin{equation}\label{eq:f_n-normalization'}
 f_n(u) = \frac{1}{\sqrt[+]{\mu_n(u)}\,}P_n\big(S f_{n-1}( u)\big),\ \  n\ge 1,
 \qquad \1 f_0( u),1\2=\sqrt[+]{\kappa_0(u)},
\end{equation}
where  $S$ denotes the shift operator $S : H^{1+s}_+\to H^{1+s}_+$, $u(x)\mapsto u(x) e^{i x}$. 
We note that for any $u \in H^s_{r,0}$, $-1/2 < s \le 0$, the infinite products in \eqref{eq:kappa_n} and \eqref{eq:mu_n} are well defined 
and converge absolutely (cf. \cite[Section 3]{GKT1}). 

\begin{Rem}\label{rem:Birkhoff_maps}
Using that for any $s >  -1/2$, the inclusion 
$$
\imath : \h^{\frac{1}{2}+s}_+\to \h^{\frac{1}{2} +s}_{r, 0}, \, 
(z_n)_{n\ge 1} \mapsto \big( (\, \overline{z_{-n}} \,)_{n \le -1}, (z_n)_{n\ge 1} \big),
$$ 
is an $\R$-linear isomorphism\footnote{We ignore the complex structure of 
$\h^{\frac{1}{2}+s}_+$ and consider the space as a real Banach space.},
it follows from \cite[Theorem 1]{GK} ($s=0$) and \cite[Theorem 6]{GKT1} ($-1/2 < s < 0$)
and \cite[Proposition 5, Appendix A]{GKT1} ($s > 0$) 
that for any $s >  -1/2$, the map 
\begin{equation}\label{eq:Phi}
 H^{s}_{r,0}\to\h^{\frac{1}{2}+s}_{r,0},\quad
u\mapsto\big(( \, \overline{\Phi_{-n}(u)} \, )_{n\le -1}, \, (\Phi_n(u))_{n\ge 1}\big) 
\end{equation}
is a homeomorphism.
For notational convenience, we denote it also by $\Phi$ and also refer to it as Birkhoff map.
\end{Rem}

\begin{Rem}\label{rem:normalizations}
We record that for $u\in H^{s}_{r,0}$, $s > -1/2$, $\kappa_n(u)>0$, $n \ge 0$, and $\mu_n(u)>0$, $n\ge 1$. 
If $u=0$ then $f_n( u)=h_n( u)= e_n$ (cf. \eqref{eq:h_n})
and $\lambda_n(u) = n$ for any $n \ge 0$, implying that
$\gamma_n(u) = 0$ for any $n \ge 1$, $\kappa_0(u)= 1$ and 
$n\kappa_n(u)=1$, $\mu_n(u)=1$ for any $n\ge 1$.
\end{Rem}

For any $-1/2 < s \le 0$, let $U\equiv U^{s}$ be an open neighborhood of zero in $H^{s}_{c,0}$, chosen so that 
the statement of Proposition \ref{prop:L-local} and Proposition \ref{prop:Psi} hold.  In the course of our argument,
we will shrink $U$ several times, but continue to denote it by $U$.
It follows from \eqref{eq:h_n} and Proposition \ref{prop:L-local} that for any $n\ge 0$ the map
\begin{equation}
U\to H^{1+s}_+,\quad u\mapsto h_n( u),
\end{equation}
is analytic. In addition, $h_n(u)\ne 0$ for any $n\ge 0$ by Proposition \ref{prop:Psi} $(ii)$.
In particular, we obtain that
\begin{equation}\label{eq:h_n-normalization}
P_n\big(S h_{n-1}(u)\big)=\nu_n(u)\, h_n(u),\qquad  n\ge 1,
\end{equation}
where $\nu_n : U\to\C$ is analytic. 
It follows from \cite[Theorem 3]{GKT2} and Proposition \ref{prop:L-local} that for any $n\ge 1$,
the infinite product in \eqref{eq:mu_n} converges absolutely for $u\in U$. Hence
$\mu_n$, $n\ge 1$, extend as analytic functions to the neighborhood $U$ of zero in $H^{s}_{c,0}$. 
By shrinking the neighborhood $U$ if needed, we then obtain from
\cite[Proposition 5]{GKT2} that
\begin{equation}\label{eq:mu_n-nonzero}
\big|\mu_n(u)-1\big|<\frac{1}{2}, \qquad  u\in U, \, n\ge 1,
\end{equation}
implying that for any $n \ge 1$, the map
\begin{equation}\label{eq:mu_n-analytic}
\sqrt[+]{\mu_n} : U\to\C
\end{equation} 
is well defined and analytic. 
By shrinking the neighborhood $U$ once more if necessary, we obtain from \cite[Corollary 6]{GKT2} that
 $\kappa_n$, $n \ge 0$, in \eqref{eq:kappa_n}  extend as analytic functions to the neighborhood $U$. 
Similarly, by shrinking $U$ further if necessary,  we obtain from
\cite[Proposition 4]{GKT2} and the fact that $\kappa_0(0)=1$ that for any $u\in U$,
\begin{equation*}
\big|\kappa_0(u)-1\big|<\frac{1}{2}\quad\text{\rm and}\quad\big|n\kappa_n(u)-1\big|<\frac{9}{10},\quad n\ge 1,
\end{equation*}
and, by \cite[Proposition 3]{GKT2}, we have the following lemma.

\begin{Lem}\label{lem:varkappa}
For any $-1/2 < s \le 0$ there exists an open neighborhood $U^s$ of zero in $H^{s}_{c,0}$ so that the map
\begin{equation}\label{eq:varkappa}
\varkappa : U^s \to \ell^\infty_+,\quad u\mapsto  (\varkappa_n(u))_{n \ge 1}, \qquad
\varkappa_n(u):= 1/\sqrt[+]{n\kappa_n(u)},
\end{equation}
is analytic.
\end{Lem}

 Proposition \ref{prop:L-local},
\eqref{eq:mu_n-nonzero}, and \eqref{eq:mu_n-analytic}, then allow us 
to extend $f_n$ to a neighborhood of zero in $H^s_{c, 0}$, $-1/2 < s \le 0$, so that the extension is analytic 
and  the normalization conditions \eqref{eq:f_n-normalization'} are satisfied. We have  the following

\begin{Lem}\label{f_n analytic}
For any $-1/2 < s \le 0$, there exists an open neighborhood $U^s$ of zero in $H^{s}_{c,0}$ so that
for any $n\ge 0$, the map $f_n : U^s \cap H^s_{r, 0} \to H^{1+s}_+$ extends to an analytic map
\begin{equation}\label{eq:f_n-map}
U^s \to H^{1+s}_+,\quad u\mapsto f_n( u).
\end{equation}
\end{Lem}

\begin{proof}[Proof of Lemma \ref{f_n analytic}]
Let $ -1/2 < s \le 0$ and choose the neighborhood $U \equiv U^s$ of zero in $H^s_{c,0}$
so that Proposition \ref{prop:L-local} and Proposition \ref{prop:Psi} hold and so that the properties, discussed above, are satisfied.
First we extend the eigenfunction $f_0$ to $U$. Since by \eqref{eq:Psi_n-map}, $h_0 : U \to H^{1+s}_+$ is analytic
and by Proposition \ref{prop:Psi}, $\langle h_0 , 1 \rangle$ does not vanish on $U$, the map
\begin{equation}\label{def a_0}
U \to \C, \, u \mapsto a_0(u):= \frac{\sqrt[+]{\kappa_0(u)}}{\langle h_0(u) , 1 \rangle},
\end{equation}
and hence $U \to H^{1+s}_+, \, u \mapsto f_0(u) := a_0(u)  h_0(u)$,
are well defined and analytic. Using that for any $n \ge 1$,
$U\to\LL\big(H^{s}_+,H^{1+s}_+\big)$, $u\mapsto P_n(u)$ (cf. \eqref{P_n anal near 0})
is analytic and $\mu_n : U \to \C$ satisfies
\eqref{eq:mu_n-nonzero} and \eqref{eq:mu_n-analytic}
one then concludes by induction that for any $n \ge 1$,
$$
U \to H^{1+s}_+, \, u \mapsto f_n(u) := \frac{1} {\sqrt[+]{\mu_n(u)}} P_n(S f_{n-1}(u)) 
$$
is analytic as well. We thus have proved that 
for any $u \in U$, 
$f_n : U \to H^{1+s}_+$, $n \ge 0$, are analytic,
satisfying $L_u f_n(u) = \lambda_n(u) f_n(u)$ and  the normalisation conditions \eqref{eq:f_n-normalization'}.
\end{proof}
Let us now study for any  $u \in U$ and $n \ge 0$ the relation between $f_n(u)$ and $h_n(u)$
where $U \equiv U^s$ is the neighborhood of zero in $H^s_{c,0}$ of Lemma \ref{f_n analytic}.
Recall that by the proof of Lemma \ref{f_n analytic},  $f_0(u) = a_0(u) h_0(u)$ where
$a_0 : U \to \C$ is the analytic map given by \eqref{def a_0}.
Let us now turn to the case $n \ge 1$.
Since $f_n( u)$ and $h_n(u)$ belong to the one dimensional (complex) eigenspace 
of $L_u$ corresponding to the simple eigenvalue $\lambda_n(u)$ and both do not vanish we have that
\begin{equation}\label{eq:f<->h}
f_n(u)=a_n(u) h_n(u),\quad n\ge 1,
\end{equation}
where the map
\begin{equation}\label{eq:a_n-map}
a_n : U\to\C
\end{equation} 
is analytic.
It now follows from \eqref{eq:f<->h}, \eqref{eq:h_n-normalization}, and the normalization conditions 
\eqref{eq:f_n-normalization'}, extended to complex valued potentials $u\in U$ as explained above, 
that for any $u\in U$ and $n\ge 1$,
\begin{eqnarray*}
\sqrt[+]{\mu_n} a_n h_n&=&\sqrt[+]{\mu_n}\,f_n=P_n\big(S f_{n-1}\big)=P_n\big(S (a_{n-1} h_{n-1})\big)\\
&=&a_{n-1} P_n\big(S h_{n-1}\big)=a_{n-1}\nu_n h_n .
\end{eqnarray*}
In view of \eqref{eq:mu_n-nonzero}, \eqref{eq:f<->h}, 
and the fact that $\1 h_0( u)|1\2\ne 0$ (cf. Proposition \ref{prop:Psi} $(ii)$),
one then infers that
\begin{equation}\label{eq:a_n}
a_n(u) =\frac{\nu_n(u)}{\sqrt[+]{\mu_n(u)}}\, a_{n-1}(u),\quad \forall \, n\ge 1.
\end{equation}
Hence, for any $u\in U$ and $n\ge 1$ we have that
\begin{equation}\label{eq:a_n-product_formula}
a_n(u)=a_0(u)\prod_{k=1}^n\frac{\nu_k(u)}{\sqrt[+]{\mu_k(u)}}
=a_0(u)\Big(\prod_{k=1}^n\nu_k(u)\Big)\Big(\prod_{k=1}^n\frac{1}{\sqrt[+]{\mu_k(u)}}\Big)\,.
\end{equation}
It follows from \eqref{eq:h_n-normalization} that for any $u\in U$ and $n\ge 1$,
\[
\nu_n(u)=\frac{\big\1 P_n\big(S h_{n-1}(u)\big)\big|e_n\big\2}{\1 h_n(u)|e_n\2}.
\]
Hence,
\begin{equation}\label{eq:nu_n}
\nu_n(u)=1+\frac{\delta_n(u)}{\alpha_n(u)}, \qquad 
\alpha_n(u):=\1 P_n e_n|e_n\2
\end{equation}
where
\begin{equation}\label{def:delta_n}
\delta_n(u):=\beta_n(u)-\alpha_n(u), \qquad
\beta_n(u):=\big\1 P_nS P_{n-1}e_{n-1}\big|e_n\big\2 .
\end{equation}
Recall from Proposition \ref{prop:Psi} $(ii)$ that 
\begin{equation}\label{eq:alpha_n-estimate}
|\alpha_n(u)|\ge 1/2
\end{equation}
for any 
$u\in U$ and $n\ge 1$.
In Section \ref{sec:the_delta_map} we prove the following important proposition, 
which will be a key ingredient into the proof of Theorem \ref{th:Phi}.

\begin{Prop}\label{prop:delta_n-analyticity}
For any $-1/2 < s \le 0$, there exists an open neighborhood $U^s$ of zero in $H^{s}_{c,0}$ so that the map
\begin{equation}\label{eq:delta}
\delta : U^s \to \ell^1_+,\quad u\mapsto\big(\delta_n(u)\big)_{n\ge 1},
\end{equation}
with $\delta_n$ defined in \eqref{def:delta_n}, is analytic and bounded.
\end{Prop}

\begin{Rem}\label{delta_n for u=0}
By Remark \ref{rem:normalizations}, $\delta_n(0) = 0$ for any $n \ge 1$.
\end{Rem}
Proposition \ref{prop:delta_n-analyticity} allows us to prove the following lemma.

\begin{Lem}\label{lem:a}
For any $-1/2 < s \le 0$ there exists an open neighborhood $U^s$ of zero in $H^{s}_{c,0}$ so that the map
\begin{equation}\label{eq:a}
a : U^s \to \ell^\infty_+,\quad u\mapsto\big(a_n(u)\big)_{n\ge 1},
\end{equation}
with $a_n$ defined by \eqref{eq:f<->h}, is analytic.
\end{Lem}

\begin{proof}[Proof of Lemma \ref{lem:a}]
For any given $ -1/2 < s \le 0$, choose a neighborhood $U \equiv U^s$ of zero in $H^s_{c, 0}$
so that the results discussed above (in particular Proposition \ref{prop:delta_n-analyticity}) hold.
First, note that by the analyticity of \eqref{eq:a_n-map}, any of the components of 
the map \eqref{eq:a} is an analytic function of $u\in U$. Hence, the lemma will follow if we prove that
the map \eqref{eq:a} is bounded (see e.g. \cite[Theorem A.3]{KP-book}). To this end, we prove that the two products 
appearing on the right side of \eqref{eq:a_n-product_formula} are bounded uniformly in $u\in U$ and $n\ge 1$.
By shrinking the neighborhood $U$ if necessary, we can ensure from 
Proposition \ref{prop:delta_n-analyticity}, Remark \ref{delta_n for u=0}, 
and \eqref{eq:alpha_n-estimate}
that there exists $C>0$ such that for any $u\in U$ and $n\ge 1$ we have that
\[
\sum_{k=1}^n|\delta_k(u)|\le C\quad\text{\rm and}\quad
\left|\frac{\delta_n(u)}{\alpha_n(u)}\right|\le\frac{1}{2}\,.
\]
This together with \eqref{eq:nu_n} and \eqref{eq:alpha_n-estimate} then implies that 
there exists $C>0$ such that for any $u\in U$ and $n\ge 1$,
\begin{equation}\label{eq:log(nu)+sum}
\sum_{k=1}^n\big|\log\nu_k(u)\big|=
\sum_{k=1}^n\left|\log\Big(1+\frac{\delta_k(u)}{\alpha_k(u)}\Big)\right| \le C ,
\end{equation}
where $\log \lambda$ denotes the standard branch of the (natural) logarithm on $\C\setminus(-\infty,0]$, defined by
$\log \lambda:=\log|\lambda|+i\arg \lambda$ and $-\pi<\arg \lambda <\pi$.
Hence, in view of \eqref{eq:log(nu)+sum}, for any $u\in U$ and $n\ge 1$ we obtain
\begin{eqnarray*}
\Big|\prod_{k=1}^n\nu_k(u)\Big|&=&
\left|\exp\Big(\sum_{k=1}^n\log\nu_k(u)\Big)\right|\\
&\le&\exp\Big(\sum_{k=1}^n\big|\log\nu_k(u)\big|\Big)\le\exp C\,.
\end{eqnarray*}
Let us now prove the boundedness of the second product on the right side of \eqref{eq:a_n-product_formula}.
It follows from \eqref{eq:mu_n-nonzero} that $\log\mu_n(u)=\log\big(1-(1-\mu_n(u))\big)$ is 
well defined for any $u\in U$ and $n\ge 1$.
By shrinking the neighborhood $U$ if necessary, we obtain from \cite[Theorem 3 and Remark 4]{GKT2} 
that 
there exists $C>0$ such that for any $u\in U$ and $n\ge 1$ we have that
\begin{equation*}
\sum_{k=1}^n\big|\log\mu_k(u)\big|\le C.
\end{equation*}
This implies that for any $u\in U$ and $n\ge 1$ we obtain
\begin{eqnarray*}
\left|\prod_{k=1}^n\frac{1}{\sqrt[+]{\mu_k(u)}}\right|&=&
\left|\exp\Big(-\frac{1}{2}\sum_{k=1}^n\log\mu_k(u)\Big)\right|\\
&\le&\exp\Big(\frac{1}{2}\sum_{k=1}^n\big|\log\mu_k(u)\big|\Big)\le\exp\big(C/2\big),
\end{eqnarray*}
which completes the proof of the lemma.
\end{proof}

\medskip

As a consequence from Lemma \ref{lem:varkappa}, Proposition \ref{prop:delta_n-analyticity}, and Lemma \ref{lem:a},
we obtain the following more general statement.

\begin{Coro}\label{coro:s>-1/2}
For any $s>-1/2$ there exists an open neighborhood $U^s$ of zero in $H^s_{c,0}$ such that the maps
\eqref{eq:varkappa}, \eqref{eq:delta}, and \eqref{eq:a} are well defined and analytic. 
The neighborhood $U^s$ can be chosen  invariant with respect to the complex conjugation of functions.
\end{Coro}

\begin{proof}[Proof of Corollary \ref{coro:s>-1/2}]
Take $s>-1/2$ and denote $\sigma\equiv\sigma(s):=\min(s,0)$. Then, it follows from Lemma \ref{lem:varkappa},
Proposition \ref{prop:delta_n-analyticity}, and Lemma \ref{lem:a} that there exists an open neighborhood
$U^\sigma$ of zero in $H^\sigma_{c,0}$ such that the maps \eqref{eq:varkappa}, \eqref{eq:delta}, and \eqref{eq:a} 
(with $U^s$ replaced by $U^\sigma$) are analytic. Denote by $U^s$ the open neighborhood $U^\sigma\cap H^s_{c,0}$ 
of zero in $H^s_{c,0}$. The analyticity of the maps \eqref{eq:varkappa}, \eqref{eq:delta}, and \eqref{eq:a} then follow from the 
boundedness of the  inclusion $U^s\hookrightarrow U^\sigma$. Finally, by taking $U^\sigma$ above to be 
an open ball in $H^\sigma_{c,0}$, centered at zero, we obtain the last statement in Corollary \ref{coro:s>-1/2}.
\end{proof}

\medskip

Our last step is to rewrite formula \eqref{eq:Phi}, which defines the Birkhoff map 
$\Phi : H^{s}_{r,0}\to\h^{\frac{1}{2}+s}_{r,0}$ for real valued $u\in H^{s}_{r,0}$, $s>-1/2$, in a form that will allow us 
to extend $\Phi$ to an analytic map in a (complex) neighborhood $U^s$ of zero in $H^{s}_{c,0}$. 
To this end, we use the symmetries of the Lax operator $L_u$ established in Corollary \ref{coro:symmetries} in 
Appendix \ref{Appendix symmetries} and argue as follows:
First,  we note that if the map $F : U\to Y$, where $U\subseteq X$ is an open neighborhood and $X$ and $Y$ are 
(complex) Banach spaces, is analytic then so is the map
\begin{equation}\label{eq:F*}
F^* : \overline{U}\to Y,\quad x\mapsto F^*(x):=\overline{F(\overline{x})},
\end{equation}
where $\overline{(\cdot)}$ denotes complex conjugation. With this in mind, we note that for any 
$u\in H^{s}_{r,0}$ with $s > -1/2$ and $n\ge 0$ we have that
\begin{equation}\label{eq:conjugation_trick}
\big\1 1,\overline{f_n(u)}\big\2=\big\1 1,\overline{f_n(\overline{u})}\big\2.
\end{equation}
For any $u\in H^{s}_{r,0}$ with $s > -1/2$, we obtain from Corollary \ref{coro:symmetries} that 
$\lambda_n(u)=\overline{\lambda_n(\overline{u})}$, $n\ge 0$.
This, together with the definition \eqref{eq:kappa_n} of $\kappa_n$ implies that  for any $u\in H^{s}_{r,0}$, $s > -1/2 $,
\begin{equation}\label{eq:kappa_n+symmetries}
\kappa_n(u)=\overline{\kappa_n(\overline{u})},\quad n\ge 0.
\end{equation}
We choose the neighborhood $U^s$ of zero in $H^s_{c,0}$, $s>-1/2$, so that Corollary \ref{coro:s>-1/2}
and Proposition \ref{prop:Psi} hold and so that it is invariant with respect to complex conjugation. 
It then follows from \eqref{eq:Phi_n}, \eqref{eq:f<->h}, \eqref{eq:conjugation_trick}, and 
\eqref{eq:kappa_n+symmetries}, that for any $u\in U^s\cap H^{s}_{r,0}$, $s>-1/2$, and $n\ge 1$,
\begin{align}
\Phi_n(u)&=\sqrt{n}\,\frac{\overline{a_n(\overline{u})}}{\sqrt[+]{n\kappa_n(u)}}\,\big\1 1,\overline{h_n(\overline{u})}\big\2\nonumber\\
&=\sqrt{n}\,\frac{a_n^*(u)}{\sqrt[+]{n\kappa_n(u)}}\,\big\1 1,h_n^*(u)\big\2\label{eq:Phi2}
\end{align}
and
\begin{equation}\label{eq:Phi1}
\overline{\Phi_n(u)}=\sqrt{n}\,\frac{a_n(u)}{\sqrt[+]{n\kappa_n(u)}}\,\big\1 1,h_n(u)\big\2,
\end{equation}
where we used the notation, introduced in \eqref{eq:F*}. Hence, the Birkhoff map \eqref{eq:Phi} satisfies for any 
$u\in U^s\cap H^{s}_{r,0}$, $s>-1/2$, 
\begin{equation}\label{eq:Phi-modified}
\Phi(u)=\Big(\big(\Phi^*_{-n}(u)\big)_{n\le -1},\big(\Phi_n(u)\big)_{n\ge 1}\Big)\in \h^{\frac12 + s}_{r,0}.
\end{equation}
Note that by the discussion above, the right hand sides of the identities in \eqref{eq:Phi2} and \eqref{eq:Phi1} are well defined 
and analytic (cf. Corollary \ref{coro:s>-1/2}) for any $u\in U^s$, $s>-1/2$, and $n\ge 1$.
The main result of this section is the following proposition.

\begin{Prop}\label{prop:Phi-analytic}
For any $s>-1/2$, there exists an invariant with respect to the complex conjugation open neighborhood $U^s$ of zero in $H^{s}_{c,0}$ 
so that the right hand side of the identity in \eqref{eq:Phi-modified} is well defined for any $u\in U^s$ and the map
\begin{equation}\label{eq:Phi_map-modified}
\Phi : U^s\to\h^{\frac{1}{2}+s}_{c,0},\quad u\mapsto\Big(\big(\Phi^*_{-n}(u)\big)_{n\le -1},\big(\Phi_n(u)\big)_{n\ge 1}\Big),
\end{equation}
is analytic.
\end{Prop}

\begin{Rem}
Since $H^{s}_{r,0}$ is a real space in the complex space $H^{s}_{c,0}$ and, similarly,
$h^{\frac12 + s}_{r,0}$ is a real space in the complex space $h^{\frac12 + s}_{c,0}$, we conclude from 
Proposition \ref{prop:Phi-analytic} that the Birkhoff map \eqref{eq:Phi} is real analytic
on $U^s\cap H^{s}_{r,0}$.
\end{Rem}

\begin{proof}[Proof of Proposition \ref{prop:Phi-analytic}]
For a given $s>-1/2$ we choose the open neighborhood $U^s$ of zero in $H^s_{c,0}$ as in Corollary \ref{coro:s>-1/2}
and Proposition \ref{prop:Psi} and assume that $U^s$ is invariant with respect to the complex conjugation.
Then, the map \eqref{eq:Phi_map-modified} is well defined on $U^s$
and takes values in $\h^{\frac{1}{2}+s}_{c,0}$. The map \eqref{eq:Phi_map-modified} is analytic if the map
\begin{equation}\label{eq:Phi^(2)}
\Phi^{(2)} : U^s \to \h^{\frac{1}{2}+s}_+,\quad u\mapsto\big(\Phi_n(u)\big)_{n\ge 1},
\end{equation}
with $\Phi_n(u)$ given by \eqref{eq:Phi2}, is analytic.
Let $\Lambda^{1/2}$ denote the linear multiplier,
$$
\Lambda^{1/2} : \h^{s}_+\to\h^{s - \frac12}_+,\quad(x_n)_{n\ge 1}\mapsto\big(\sqrt{n}\,x_n\big)_{n\ge 1} .
$$
The map $\Phi^{(2)}$ is then represented by the commutative diagram
\[
\begin{tikzcd}
U^s \arrow[d, swap,"{(a^*,\varkappa,\Psi^*)}"]\arrow[drr, bend left, "\Phi^{(2)}"]    &                                              &\\
\ell^\infty_+\times\ell^\infty_+\times \h^{1+s}_+\arrow[r,"M"]&\h^{1+s}_+\arrow[r,"\Lambda^{1/2}"]&\h^{\frac{1}{2}+s}_+\\
\end{tikzcd}
\]
where 
\begin{align}
&(a^*,\varkappa,\Psi^*) : U^s \to\ell^\infty_+\times\ell^\infty_+\times \h^{1+s}_+,\quad
u\mapsto\big(a^*(u),\varkappa(u),\Psi^*(u)\big),\label{eq:vertical_arrow}\\
&M : \ell^\infty_+\times\ell^\infty_+\times \h^{1+s}_+\to\h^{1+s}_+,\quad
\big((x_n)_{n\ge 1}, (y_n)_{n\ge 1},(z_n)_{n\ge 1}\big)\mapsto\big(x_ny_nz_n\big)_{n\ge 1},\label{eq:M}
\end{align}
and $a : U^\sigma \to\ell^\infty_+$, $\varkappa : U^\sigma \to\ell^\infty_+$, and $\Psi : U^s \to \h^{1+s}_+$ are given
respectively in \eqref{eq:a}, \eqref{eq:varkappa}, and \eqref{eq:Psi}.
By Corollary \ref{coro:s>-1/2} and Proposition \ref{prop:Psi} the map \eqref{eq:vertical_arrow} is analytic, 
whereas the map \eqref{eq:M} is analytic since it is a bounded (complex) trilinear map.
This, together with the analyticity of the multiplier operator $\Lambda^{1/2}$ and 
the commutative diagram above implies that \eqref{eq:Phi^(2)} is analytic. This completes the proof of the proposition.
\end{proof}


\section{Analyticity of the delta map}\label{sec:the_delta_map}
In this section we prove Proposition \ref{prop:delta_n-analyticity} of Section \ref{sec:the_Birkhof_map},
which plays a significant role in the proof of Proposition \ref{prop:Phi-analytic}.
Our proof of Proposition \ref{prop:delta_n-analyticity} is based on a vanishing lemma -- see Lemma \ref{lem:D=0} below.
Throughout this section we assume that $-1/2 < s \le 0$.. 

First we make some preliminary considerations.
Assume that $-1/2 < s \le 0$ and let $U\equiv U^{s}$ be an open neighborhood of zero in $H^{s}_{c,0}$ chosen so that 
the statement of Proposition \ref{prop:L-local} holds. 
Then, the Lax operator $L_u = D-T_u$ is a closed operator in $H^{s}_+$ with domain $H^{1+s}_+$, 
has a compact resolvent and its spectrum consists of countably many simple eigenvalues, 
$\lambda_n\in D_n(1/4)$, $n\ge 0$. Recall that the {\em $\delta$-map} is defined on $U$ and for any $u \in U$, one has
\begin{equation}\label{def1:delta_n}
\delta(u) = \big(\delta_n(u)\big)_{n\ge 1}, \qquad \delta_n(u) = \beta_n(u)-\alpha_n(u) ,
\end{equation}
where $\beta_n(u)=\big\1 P_nS P_{n-1}e_{n-1}\big|e_n\big\2$ (cf. \eqref{def:delta_n}) and 
$\alpha_n(u)= \1 P_n e_n|e_n\2$ (cf. \eqref{eq:nu_n}). 

It follows from Proposition \ref{prop:L-local} and Remark \ref{rem:shifted_norm} that the Neumann series 
expansion of the resolvent of $L_u$,
\begin{equation}
(L_u-\lambda)^{-1}=\sum_{m\ge 0}(D-\lambda)^{-1}\big[T_u(D-\lambda)^{-1}\big]^m,
\end{equation}
converges uniformly in $\LL\big(H^{s}_+\big)$ for $(u,\lambda)\in U\times\bigcup_{n\ge 0}\Wert_n^0(1/4)$. 
By the definition of $P_{n-1}$ in \eqref{P_n anal near 0}, this implies that $\beta_n \equiv \beta_n(u)$, $n\ge 1$, 
satisfies
\begin{equation}\label{eq:beta_n}
\beta_n  =-\sum_{m\ge 0}\frac{1}{2\pi i}\oint\limits_{\partial D_{n-1}}
\big\1 P_nS(D-\lambda)^{-1}\big[T_u(D-\lambda)^{-1}\big]^me_{n-1}\big|e_n\big\2\,d\lambda ,
\end{equation}
where the series converges absolutely and where $D_{n-1} = D_{n-1}(1/3)$ and $\partial D_{n-1}$
is counterclockwise oriented. The first term in the latter series is
\begin{eqnarray}
-\frac{1}{2\pi i}\!\!\!\!\oint\limits_{\partial D_{n-1}}\!\!\!\big\1 P_nS(D-\lambda)^{-1}e_{n-1}\big|e_n\big\2\,d\lambda
&=&\big\1 P_nSe_{n-1}\big|e_n\big\2\frac{1}{2\pi i}\!\!\!\!\oint\limits_{\partial D_{n-1}}\!\!\!
\frac{d\lambda}{\lambda-(n-1)}\nonumber\\
&=&\big\1 P_ne_n\big|e_n\big\2=\alpha_n  \label{eq:alpha_n},
\end{eqnarray}
where we used that $S e_{n-1}=e_n$. 
By combining \eqref{def1:delta_n} with \eqref{eq:beta_n} and \eqref{eq:alpha_n} 
we obtain that for $n\ge 1$,
\begin{eqnarray}\label{eq:delta_n}
\delta_n \equiv \delta_n(u)&=&-\sum_{m\ge 1}\frac{1}{2\pi i}\oint\limits_{\partial D_{n-1}}
\big\1 P_nS(D-\lambda)^{-1}\big[T_u(D-\lambda)^{-1}\big]^me_{n-1}\big|e_n\big\2\,d\lambda\nonumber\\
&=&-\sum_{m\ge 1}\sum_{\tb{k_j\ge 0}{1\le j\le m}}C_u(k_1,...,k_m)\,
\1 P_nSe_{k_m}|e_n\2
\end{eqnarray}
where
\[
C_u(k_1,...,k_m):=\frac{1}{2\pi i}\!\!\!\oint\limits_{\partial D_{n-1}}\!\!\!\frac{\hu(k_1-(n-1))}{(n-1)-\lambda}
\frac{\hu(k_2-k_1)}{k_1-\lambda}\cdots\frac{\hu(k_m-k_{m-1})}{k_{m-1}-\lambda}\frac{d\lambda}{k_m-\lambda}.
\]
By passing to the variable $\mu:=\lambda-(n-1)$ in the contour integral above and then setting
$l_j:=k_j-(n-1)$, $1\le j\le m$, we obtain from \eqref{eq:delta_n} that
\begin{equation}\label{eq:first_sum}
\delta_n=\sum_{m\ge 1}\sum_{\tb{l_j\ge -n+1}{1\le j\le m}}A(l_1,...,l_m)B_u(l_1,...,l_m)\,
\1 P_n e_{l_m+n}|e_n\2
\end{equation}
where
\begin{equation}\label{eq:A}
A(l_1,...,l_m):=\frac{1}{2\pi i}\oint\limits_{\partial D_0}\frac{1}{\mu}\prod_{j=1}^m\frac{1}{l_j-\mu}\,d\mu,
\end{equation}
and
\begin{equation}\label{eq:B}
B_u(l_1,...,l_m):=\left\{
\begin{array}{l}
\hu(l_1)\quad\text{\rm if}\quad m=1,\\
\hu(l_1)\hu(l_2-l_1)\cdots\hu(l_m-l_{m-1})\quad\text{\rm if}\quad m\ge 2.
\end{array}
\right.
\end{equation}
Consider the term $\1 P_n e_{l_m+n}|e_n\2$ that appears in \eqref{eq:first_sum}. It follows from
\eqref{eq:neumann_series} and the formula for the Riesz projector that
\begin{eqnarray*}
\1 P_n e_{l_m+n}|e_n\2&=&-\sum_{r\ge 0}\frac{1}{2\pi i}\oint\limits_{\partial D_n}
\big\1(D-\lambda)^{-1}\big[T_u(D-\lambda)^{-1}\big]^re_{l_m+n}\big|e_n\big\2\,d\lambda.\nonumber
\end{eqnarray*}
For the first term in the latter series we have
\begin{equation}\label{eq:r=0}
-\frac{1}{2\pi i}\oint\limits_{\partial D_n}
\big\1(D-\lambda)^{-1}e_{l_m+n}\big|e_n\big\2\,d\lambda=\1 e_{l_m+n}|e_n\2
\frac{1}{2\pi i}\oint\limits_{\partial D_n}\frac{d\lambda}{\lambda-(l_m+n)}=\delta_{l_m0}.
\end{equation}
Hence
\begin{equation}\label{eq:second_sum}
\1 P_n e_{l_m+n}|e_n\2 =\delta_{l_m0}+\sum_{r\ge 1}\sum_{\tb{k_j\ge 0}{1\le j\le r-1}}C_u(k_1,...,k_{r-1})
\end{equation}
where
\[
C_u(k_1,...,k_{r-1}):=\left\{
\begin{array}{l}
-\frac{1}{2\pi i}\oint\limits_{\partial D_n}\!\!\!\frac{\hu(-l_m)}{(l_m+n)-\lambda}\frac{d\lambda}{n-\lambda}
\quad\text{\rm if}\quad r=1,\\
-\frac{1}{2\pi i}\oint\limits_{\partial D_n}\!\!\!\frac{\hu(k_1-(l_m+n))}{(l_m+n)-\lambda}
\frac{\hu(k_2-k_1)}{k_1-\lambda}\cdots\frac{\hu(n-k_{r-1})}{k_{r-1}-\lambda}\frac{d\lambda}{n-\lambda}
\,\,\text{\rm if }r\ge 2.
\end{array}
\right.
\]
By passing to the variable $\mu:=\lambda-n$ in the contour integral and then setting $l_{m+j}:=k_j-n$, 
$1\le j\le r-1$, we obtain from \eqref{eq:second_sum} that
\begin{eqnarray}\label{eq:second_sum'}
\1 P_n e_{l_m+n}|e_n\2&=&\delta_{l_m0}+\sum_{r\ge 1}\sum_{\tb{l_j\ge -n}{m+1\le j\le m+r-1}}
A(l_m,...,l_{m+r-1})B'_u(l_m,...,l_{m+r-1})\nonumber\\
&=&\delta_{l_m0}+\sum_{r\ge 0}\sum_{\tb{l_j\ge -n}{m+1\le j\le m+r}}
A(l_m,...,l_{m+r})B'_u(l_m,...,l_{m+r})
\end{eqnarray}
where
\begin{equation}\label{eq:B'}
B'_u(l_m,...,l_{m+r}):=\left\{
\begin{array}{l}
\hu(l_{m+1}-l_m)\cdots\hu(l_{m+r}-l_{m+r-1})\hu(-l_{m+r})\,\ \text{ \rm if}\,\,r\ge 1 , \\
\hu(-l_m)\quad\text{\rm if}\quad r=0 .
\end{array}
\right.
\end{equation}

In view of the expansion \eqref{eq:second_sum'}, we split $\1 P_n e_{l_m+n}|e_n\2$,
with one of the terms being $\delta_{l_m0}$, and then split $\delta_n$ accordingly.
To this end, it turns out to be useful to introduce  
\begin{equation}\label{eq:E}
E_u(l_1,...,l_d):=\left\{
\begin{array}{l}
\hu(l_1)\hu(-l_1)\quad\text{\rm if}\quad d=1,\\
\hu(l_1)\hu(l_2-l_1)\cdots\hu(l_d-l_{d-1})\hu(-l_d)
\quad\text{\rm if}\quad d\ge 2 
\end{array}
\right.
\end{equation}
By \eqref{eq:B} and \eqref{eq:B'} one has
\[
B_u(l_1,...,l_m)B_u'(l_m,...,l_{m+r})=E_u(l_1,...,l_{m+r}).
\]
By \eqref{eq:first_sum} and \eqref{eq:second_sum'}, we then have
\begin{equation}\label{eq:delta_n+split}
\delta_n=\delta_n^{(1)}+\delta_n^{(2)}
\end{equation}
where
\begin{eqnarray}\label{eq:delta_n'}
\delta_n^{(2)}&:=&\sum_{m\ge 1}\sum_{\tb{l_j\ge -n+1}{1\le j\le m-1}}A(l_1,...,l_{m-1},0)B_u(l_1,...,l_{m-1},0)\nonumber\\
&=&-\sum_{d\ge 1}\sum_{\tb{l_j\ge -n+1}{1\le j\le d}}A(l_1,...,l_d, 0)E_u(l_1,...,l_d)
\end{eqnarray}
and
\begin{align*}
\delta_n^{(1)}:=\sum_{m\ge 1}\sum_{r\ge 0}\sum_{\tb{l_j\ge -n+1}{1\le j\le m}}\sum_{\tb{l_j\ge -n}{m+1\le j\le m+r}}
A(l_1,...,l_m)A(l_m,...,l_{m+r})E_u(l_1,...,l_{m+r})\\
=\sum_{\tb{d\ge 1}{1\le m\le d}}\sum_{\tb{l_j\ge -n+1}{1\le j\le m}}\sum_{\tb{l_j\ge -n}{m+1\le j\le d}}
A(l_1,...,l_m)A(l_m,...,l_{d})E_u(l_1,...,l_{d}) .
\end{align*}
Since the range of $l_j$, $1 \le j \le m$, and the one of $l_j$, $m+ 1 \le j \le d$, are different,
we split the latter sum into two parts
\begin{align*}
\delta_n^{(1)}=\sum_{d\ge 1}\sum_{\tb{l_j\ge -n+1}{1\le j\le d}}
\Big(\sum_{1\le m\le d}A(l_1,...,l_m)A(l_m,...,l_{d})\Big)E_u(l_1,...,l_{d})\\
+\sum_{\tb{d\ge 2}{1\le m\le d-1}}\sum_{\tb{l_j\ge -n+1}{1\le j\le m}}
\sum_{m+1\le k\le d}
\sum_{\tb{l_k=-n}{\tb{l_j\ge -n+1, m+1\le j<k}{l_j\ge -n, k<j\le d }}}
\!\!\!\!A(l_1,...,l_m)A(l_m,...,l_{d})E_u(l_1,...,l_{d}).
\end{align*}
By combining this with \eqref{eq:delta_n'} we conclude from \eqref{eq:delta_n+split} that
\begin{equation}\label{eq:delta_n''}
\delta_n=\sum_{d\ge 1}\sum_{\tb{l_j\ge -n+1}{1\le j\le d}}\D(l_1,...,l_d)E_u(l_1,...,l_d)+\RR_n(u)
\end{equation}
where by \eqref{eq:A},
\begin{align}\label{eq:D}
\D &(l_1,...,l_d):= \big( \sum_{1\le m\le d}A(l_1,...,l_m)A(l_m,...,l_{d}) \big)   -   A(l_1,...,l_d, 0) \nonumber\\
&=\sum_{1\le m\le d} \Big(\frac{1}{2\pi i}\oint_{\partial D_0}\frac{1}{\mu}\prod_{j=1}^m\frac{1}{l_j-\mu}\,d\mu\Big)
\Big(\frac{1}{2\pi i}\oint_{\partial D_0}  \frac{1}{\mu}\prod_{j=m}^d\frac{1}{l_j-\mu}\,d\mu\Big)\nonumber\\
& \qquad -\frac{1}{2\pi i}\oint_{\partial D_0}\frac{1}{\mu^2}\prod_{j=1}^d\frac{1}{l_j-\mu}\,d\mu,
\end{align}
 and where the remainder $\RR_n(u)$ equals
\begin{align}\label{eq:R_n}
\sum_{\tb{d\ge 2}{1\le m\le d-1}}\sum_{\tb{l_j\ge -n+1}{1\le j\le m}}
\sum_{m+1\le k\le d}
\sum_{\tb{l_k=-n}{\tb{l_j\ge -n+1, \, m+1\le j<k}{l_j\ge -n, \, k<j\le d}}}
\!\!\!\!A(l_1,...,l_m)A(l_m,...,l_{d})E_u(l_1,...,l_{d}).
\end{align}

We have the following important Vanishing Lemma.

\begin{Lem}\label{lem:D=0}
$\D(l_1,...,l_d)=0$ for any $l_1,...,l_d\in\Z$, $d \ge 1$.
\end{Lem}
\begin{proof}[Proof of Lemma \ref{lem:D=0}]
Let $d \ge 1$ be given. We show a slightly stronger result than the one
claimed by the lemma. Note that
$\D(l_1,...,l_d)$ is well defined for any
$(\ell_1, \dots \ell_d)\in (\C\setminus \overline{D_0})\cup \{ 0\}$. We prove that
\begin{equation}\label{identity for mathcal D}
\mathcal D(\ell_1,\dots ,\ell_d) = 0, \qquad \forall \, (\ell_1, \dots \ell_d)\in (\C\setminus \overline{D_0})\cup \{ 0\}.
\end{equation}
For any given $(\ell_1, \dots \ell_d)\in (\C\setminus \overline{D_0})\cup \{ 0\}$, set
$$
J:=\{ j\in \{1,\dots, d\}  \, \big| \, \ell_j=0\} ,\qquad  K:=\{ k\in \{1,\dots,d\} \,  \big| \, \ell_k\ne 0\}  .
$$
In the case where $K = \emptyset$, $I:=  \sum_{1\le m\le d}A(l_1,...,l_m)A(l_m,...,l_{d}) $
and  $II:=A^{(2)}(l_1,...,l_d)$ both vanish by the residue theorem and hence by the definition \eqref{eq:D},
the claimed identity \eqref{identity for mathcal D} holds.
For the remaining part of the proof we assume that $K \ne \emptyset$.
Notice that the terms $I$ and $II$ are holomophic functions of the 
variable $(\ell_k)_{k\in K}\in (\C\setminus \overline{D_0})^K $, so we may assume that the $\ell_k , k\in K$, are pairwise distinct
complex numbers in $\C\setminus \overline{D_0}$. By Cauchy's theorem  and Leibniz's rule, $A^{(2)}(\ell_1,\dots ,\ell_d)$ can be 
computed as
$$
\frac{(-1)^{|J|}}{(|J|+1)!}\Big[\partial_\mu^{|J|+1}\prod_{k\in K}\frac{1}{\ell_k-\mu}\Big]_{\vert \mu =0} 
=(-1)^{|J|} \sum_{{\bf q}\in Q(K,|J|+1)}\prod_{k\in K}\frac{1}{\ell_k^{q_k+1}} ,
$$
where for any integer $p \ge 1$, we set
$$
Q(K,p):=\Big\{ {\bf q} =  (q_k)_{k\in K}\in \Z_{\ge 0}^{K}\,\Big|\,\sum_{k\in K}q_k=p\Big\} .
$$
Given $m\in \{1,\dots ,d\}$, define
$$J_m:=J\cap [1,m],\ \   J'_m:=J\cap [m,d],
\qquad
K_m:=K\cap [1,m], \ \  K'_m:=K\cap [m,d].
$$
One then obtains, in a similar way,
$$
\begin{aligned}
A(\ell_1,\dots ,\ell_m)&=&(-1)^{|J_m|}\sum_{{\bf r}\in Q(K_m,|J_m|)}\prod_{k\in K_m}\frac{1}{\ell_k^{r_k+1}} ,\\
A(\ell_m,\dots ,\ell_d)&=&(-1)^{|J_m'|}\sum_{{\bf s}\in Q(K_m',|J_m'|)}\prod_{k\in K_m'}\frac{1}{\ell_k^{s_k+1}} .
\end{aligned}
$$
Since
\begin{equation}\label{J_m, J_m'}
|J_m|+|J_m'|=\begin{cases} |J|+1& {\rm if}\ m\in J \\
|J|& {\rm if}\ m\in K \end{cases} 
\end{equation}
one infers that $\sum_{m=1}^dA(\ell_1,\dots ,\ell_m)A(\ell_m,\dots, \ell_d)$ equals
$$
\begin{aligned}
&&(-1)^{|J|} \sum_{m\in K}\Big( \sum_{{\bf r}\in Q(K_m,|J_m|)}\prod_{k\in K_m}\frac{1}{\ell_k^{r_k+1}}\Big) 
\Big(\sum_{{\bf s}\in Q(K_m',|J_m'|)}\prod_{k\in K_m'}\frac{1}{\ell_k^{s_k+1}}\Big)\\
&&-(-1)^{|J|}\sum_{m\in J}\Big( \sum_{{\bf r}\in Q(K_m,|J_m|)}\prod_{k\in K_m}\frac{1}{\ell_k^{r_k+1}}\Big)
\Big(\sum_{{\bf s}\in Q(K_m',|J_m'|)}\prod_{k\in K_m'}\frac{1}{\ell_k^{s_k+1}}\Big) .
\end{aligned}
$$
For any $m\in \{1,\dots ,d\}$, ${\bf r}\in Q(K_m,|J_m|)$, and ${\bf s}\in Q(K_m',|J_m'|)$,
the corresponding term in the latter expression is of the form
$\prod_{k\in K}\frac{1}{\ell_k^{q_k+1}}$, where
$$
q_k :=r_k \ ( k<m) , \qquad q_k :=s_k \   ( k>m) ,
$$
and in case $m \in K$, $q_m :=r_m+s_m+1$. It then follows from \eqref{J_m, J_m'}
that $(q_k)_{k\in K}$ belongs to $Q(K,|J|+1)$. 
In order to describe $\sum_{m=1}^dA(\ell_1,\dots ,\ell_m)A(\ell_m,\dots, \ell_d)$ in more detail,
we define for ${\bf q}$ in $Q(K,|J|+1)$,
$$
\begin{aligned}
J_{ad}({\bf q})&:=\Big\{ m\in J \, \Big| \,  \sum_{k\in K_m}q_k=|J_m|\Big\} ,\\
\ K_{ad}({\bf q})&:= \Big\{m\in K \,\Big| \  \sum_{k\in K_m\setminus \{m\}}q_k\leq |J_m|\,,
\ \  \sum_{k\in K_m'\setminus \{m\}}q_k\leq |J_m'|\Big\}  .
\end{aligned}
$$
We remark that $J_{ad}({\bf q})$ might be empty and that by \eqref{J_m, J_m'},
for any $m\in J_{ad}({\bf q})$ and ${\bf q} \in Q(K,|J|+1)$,
the identity $\sum_{k\in K_m'}q_k=|J_m'|$ is automatically satisfied.

In view of these definitions, one has
$$
\begin{aligned}
\sum_{m=1}^d &A(\ell_1,\dots ,\ell_m)A(\ell_m,\dots, \ell_d) \\
& =  (-1)^{|J|}\sum_{{\bf q}\in Q(K,|J|+1)}
(|K_{ad}({\bf q})| -|J_{ad}({\bf q})|)\prod_{k\in K}\frac{1}{\ell_k^{q_k+1}}\ 
\end{aligned}
$$
and \eqref{identity for mathcal D} follows from the following combinatorial statement,
\begin{equation}\label{combi}
 |K_{ad}({\bf q})| =|J_{ad}({\bf q})|+1 , \qquad \forall  \, {\bf q}\in Q(K,|J|+1) .
\end{equation}
(Recall that we consider the case $(\ell_1, \dots \ell_d)\in (\C\setminus \overline{D_0})\cup \{ 0\}$ where $K \ne \emptyset$.)
To prove identity \eqref{combi}, let us fix  ${\bf q}\in Q(K,|J|+1) $. 
For notational convenience, we write $J_{ad}$, $K_{ad}$ instead of $J_{ad}({\bf q})$, $K_{ad}({\bf q})$. 
Furthermore, define 
$$S(E):=\sum_{k\in E}q_k\,, \qquad  E\subseteq K.$$

\smallskip

In a first step we prove that $K_{ad}\ne \emptyset$. (Recall that we assume $K \ne \emptyset$.)
Denote by $\underline m$ is the smallest element of $K$. 
Then $S(K_{\underline m}\setminus \{\underline m\}) =0\leq |J_{\underline m}|$.
Hence the largest number among all  $m$ in $K$, satisfying $S(K_m\setminus \{m\})\leq |J_m|$,
exists. We denote it by  $\overline m$ and claim that $\overline m \in K_{ad}$, i.e., that 
$ S(K'_{\overline m}\setminus \{ \overline m\}) \leq |J'_{\overline m}|$.
Indeed, this clearly holds if $\overline m$ is the largest element of $K$. 
Otherwise, let $j$ be the successor of $\overline m$ in $K$. Then
$$S(K_{\overline m })=S(K_j \setminus \{ j\})\geq |J_j|+1  ,$$
and  $S(K'_{\overline m}\setminus \{ \overline m\}) =(|J|+1)-S(K_{\overline m}) $ 
can be estimated  by \eqref{J_m, J_m'} as
$$
S(K'_{\overline m}\setminus \{ \overline m\}) \leq (|J|+1)-(|J_j|+1)=|J'_j|  \leq |J'_{\overline m}|  .
$$
We thus have proved that $\overline m \in K_{ad}$. Using the same type of arguments, one verifies
that the following properties are satisfied.
\begin{itemize}
\item[(P1)] $\forall \, n\in J_{ad}$, \ $\exists \, m_-, m_+\in K_{ad}$ with $m_- < n < m_+$. 
\item[(P2)] $\forall \, n_1,n_2\in J_{ad}$ with $n_1<n_2$, \ $\exists \, m\in K_{ad}$ with $n_1< m < n_2$. 
\item[(P3)] $\forall \,  m_1,m_2\in K_{ad}$ with $m_1 < m_2$, \ $\exists n \in J_{ad}$ with $m_1< n <m_2$. 
\end{itemize}
Property (P1) is proved by verifying that $m_-$, defined as the largest $m \in K$ satisfying $m < n $ and 
$S(K_m\setminus \{ m\})\leq |J_m|$, and that $m_+$, defined as the smallest $m \in K$ satisfying 
$m > n $ and $S(K_m'\setminus \{ m\} )\leq |J'_m|$, are both in $K_{ad}$. 
In more detail, one argues as follows. To prove that $m_- \in  K_{ad}$, 
recall that $\underline m \in K$, introduced in step 1, satisfies $S(K_{\underline m} \setminus \{ \underline m \}) = 0$.
Hence the largest number of all $m$ in $K$, satisfying $S(K_m\setminus \{m\}) \leq |J_m|$ and $m < n$, exists.
We denote it by $m_-$ and claim that $m_- \in  K_{ad}$. For this to be true, it remains to verify that
$S(K'_{ m_-} \setminus \{  m_- \})  \le |J_{m_-}'|$.  
If $m_-$ is the largest element in $K_n$, then
$$
S(K'_{ m_-} \setminus \{  m_- \}) = S(K'_n) = |J_n'| \le |J'_{m_-}|  ,
$$
implying that $m_- \in K_{ad}$.
Otherwise, let $j$ be the successor of $m_-$ in $K$. Since 
by the definition of $m_-$, $S(K_{m_-}) = S(K_j \setminus \{ j \} ) \ge |J_j| + 1$
one concludes
$$ 
S(K'_{ m_-} \setminus \{  m_- \}) = (|J| + 1) - S(K_{m_-})
\le  (|J| + 1) - ( |J_j| + 1) = | J'_j| \le |J'_{m_-}|  ,
$$
implying that $m_- \in K_{ad}$. (Here we used that by \eqref{J_m, J_m'}, $|J| - |J_j| = |J_j'|$.)

To prove that $m_+ \in K_{ad}$, recall that $\overline m$, introduced in step 1, is in $K_{ad}$. In particular, one has
$S(K'_{\overline m} \setminus \{ \overline m \}) \le |J'_{\overline m}|$.
Hence the smallest number among all $m$ in $K$, satisfying $S(K'_m\setminus \{m\}) \leq |J'_m|$ and $n < m$, exists.
We denote it by $m_+$. If $m_+$ is the smallest element in $K_n'$, then
$$
S(K_{ m_+} \setminus \{  m_+ \}) = S(K_n) = |J_n| \le |J_{m_+}| ,
$$
implying that $m_+  \in K_{ad}$.
Otherwise let $j$ be the predecessor of $m_+$ in K. By the definition of $m_+$,
$ S(K'_{ m_+})=S(K_{j} \setminus \{ j \}) \ge |J_j'| + 1$ and thus
$$
S(K_{ m_+} \setminus \{  m_+ \}) = (|J| + 1) -  S(K'_{ m_+}) \le  (|J| + 1) -   (|J'_j| + 1) \le |J_j| \le |J_{m_+}|,
$$
implying that $m_+ \in K_{ad}$ also in this case.

Similarly, (P2) is proved by verifying that $m$, defined as the largest $p \in K$ satisfying $n_1< p < n_2$ and $S(K_p \setminus \{ p\})\leq |J_p|$,
is in $K_{ad}$. In more detail, one argues as follows. First we verify that $(n_1, n_2)\cap K \ne \emptyset$. Indeed, otherwise one has
$$
|J| + 1 = S(K_{n_1}) +  S(K'_{n_1}) = S(K_{n_2}) +  S(K'_{n_1}) = |J_{n_2}| + |J'_{n_1}|,
$$
implying that $ |J| + 1 \ge |J_{n_1}| + 1 + |J'_{n_1}| = |J| + 2$, which is a contradiction.
Denote by $m_-$ the smallest element in $(n_1, n_2)\cap K$. Since
$S(K_{n_1}) = |J_{n_1}|$ it then follows that
$$
S(K_{ m_-} \setminus \{  m_- \}) = S(K_{n_1})  = |J_{n_1}| \le |J_{m_-}|.
$$
Hence the largest number among all $m$ in $K$, satisfying $S(K_m\setminus \{m\}) \leq |J_m|$ and $m < n_2$, exists.
We denote it by $m_+$ and claim that $m_+ \in  K_{ad}$. 
If $m_+$ is the maximal element in $K_{n_2}$, then
$$
S(K'_{ m_+} \setminus \{  m_+ \}) = S(K_{n_2})  = |J'_{n_2}| \le |J'_{m_+}|
$$
and hence $m_+ \in K_{ad}$. Otherwise, let $j$ be the successor of $m_+$ in $K_{m_2}$.
Then by the definition of $m_+$, $S(K_{ j} \setminus \{ j \}) \ge |J_j| + 1$. Therefore
$$
S(K'_{ m_+} \setminus \{  m_+ \}) = (|J| + 1) - S(K_{m_+})  
\le ( |J| + 1) - S(K_{j} \setminus \{j\})
$$
yields
$$
S(K'_{ m_+} \setminus \{  m_+ \})
\le (|J| + 1) - (|J_j| + 1) = |J'_j| 
\le |J'_{m_+}| .
$$
Hence also in this case, $m_+ \in K_{ad}$.

It remains to prove (P3). Let $m_1,m_2\in K_{ad}$ with $m_1<m_2$. First we observe that $(m_1,m_2)\cap J \ne  \emptyset$. 
Otherwise, $ |J_{m_2}| =  |J_{m_1}|$, implying that
$$
\begin{aligned}
|J|+1&=S(K_{m_1})+S(K_{m_1}'\setminus \{ m_1\})
\leq S(K_{m_2}\setminus \{ m_2\})+S(K_{m_1}'\setminus \{ m_1\})\\
&\leq  |J_{m_2}|+|J'_{m_1}|=|J_{m_1}|+|J_{m_1}'|=|J| , 
\end{aligned}
$$ 
which is a contradiction. Denote by  $\underline n$ the smallest element in $J$ which is larger than $m_1$. We then have three alternatives.
\begin{itemize}
\item[(A1)] $S(K_{\underline n})=|J_{\underline n}|$.  Then $\underline n\in J_{ad}$ and (P3) holds with $n:= \underline n$.
\item[(A2)] $S(K_{\underline n})\geq |J_{\underline n}|+1$. Since $\underline n \in J$ 
and $|J| + 1 - |J_{\underline n}| = |J'_{\underline n}|$ (cf. \eqref{J_m, J_m'}), one has
$$
S(K_{\underline n}') =  (|J|+1) - S(K_{\underline n}) \leq (|J|+1)-(|J_{\underline n}|+1)=|J'_{\underline n}|-1.
$$ 
Hence the largest number among the $n\in J\cap(m_1,m_2)$ satisfying $S(K_n')\leq |J'_n|$ exists. We denote it by $\overline n$. 
In the case $(\overline n, m_2)\cap J = \emptyset$, one has  $|J_{m_2}|=|J_{\overline n}|$ and hence
$$
S(K_{\overline n})\leq S(K_{m_2}\setminus \{ m_2\} )\leq |J_{m_2}|=|J_{\overline n}|
$$
and thus $\overline n \in J_{ad}$. 
In the case $(\overline n, m_2)\cap J \ne \emptyset$, denote by $j$ the smallest element of $J\cap(\overline n, m_2)$. 
Since by the definition of $\overline n$, $S(K'_j)\geq |J'_j|+1$ and  $|J|+1- |J'_j| = |J_j| = |J_{\overline n}| + 1$ one has
$$
S(K_{\overline n})\leq S(K_j)\leq (|J|+1)-(|J'_j|+1)=|J_j|-1=|J_{\overline n}|,$$
and we conclude again that $ \overline n \in J_{ad}$. 

\item[(A3)] $S(K_{\underline n})\leq |J_{\underline n}|-1$.  
We claim that this case does not occur. Indeed, 
$$
S(K'_{\underline n})\le S(K'_{m_1}\setminus \{m_1\}) \le |J'_{m_1}| = |J'_{\underline n}|.
$$
Since $ |J_{\underline n}| + |J'_{\underline n}| = |J|+1$ (cf. \eqref{J_m, J_m'}) it then follows that
$$
|J|+1=S(K_{\underline n})+S(K'_{\underline n})\le|J_{\underline n}|-1+|J'_{\underline n}|=|J|,
$$
which is a contradiction.
\end{itemize}
This proves (P3).

Finally, properties (P1), (P2), (P3) above together with the fact that $K_{ad}$ is not empty, imply identity \eqref{combi}.
\end{proof}


We are now ready to prove Proposition \ref{prop:delta_n-analyticity}.

\begin{proof} [Proof of Proposition \ref{prop:delta_n-analyticity}]
For any given $ -1/2 < s \le 0$, choose $U\equiv U^s$ as at the beginning of this section.
 Lemma \ref{lem:D=0} and \eqref{eq:delta_n''} imply that
for any $u\in U$,
\begin{equation}\label{eq:delta_n-final}
\delta_n(u)=\RR_n(u),\quad n\ge 1.
\end{equation}
Moreover, for any $n\ge 1$ and for any $u\in U$ we obtain from \eqref{eq:R_n} that
\begin{align*}
\sum_{n\ge 1}|\RR_n(u)|\le\sum_{n\ge 1}\sum_{\tb{d\ge 2}{1\le m\le d-1}}
\sum_{m+1\le k\le d}
\sum_{\tb{l_k=-n}{\tb{-\infty<l_j<\infty}{j\in\{1,...,d\}\setminus\{k\}}}}
\!\!\!\!\!\!|A(l_1,...,l_m)| |A(l_m,...,l_d)| |E_u|\\
\le\sum_{d\ge 2}\sum_{\tb{-\infty<l_j<\infty}{1\le j\le d}}
\Big(\sum_{1\le m\le d-1}(d-m)
|A(l_1,...,l_m)| |A(l_m,...,l_d)|\Big) |E_u|
\end{align*}
where we write $|E_u|$ for $|E_u(l_1,...,l_d)|$.
Proposition \ref{prop:delta_n-analyticity} now follows from \eqref{eq:delta_n-final} and Lemma \ref{lem:R_n-convergence}, 
Lemma \ref{lem:l^1-analyticity}, stated below.
\end{proof}

\begin{Lem}\label{lem:R_n-convergence}
For any $-1/2 < s \le 0$ there exists an open neighborhood $ U^s$ of zero in $H^s_{c,0}$ and a constant $C>0$ so that 
for any $u\in U^s$
\begin{equation}\label{eq:R_n-convergence}
\sum_{d\ge 2}\sum_{\tb{-\infty<l_j<\infty}{1\le j\le d}}
\Big(\sum_{1\le m\le d-1}(d-m)
|A(l_1,...,l_m)| |A(l_m,...,l_d)|\Big) |E_u|\le C\|u\|_{s}^3.
\end{equation}
\end{Lem}

\begin{proof}[Proof of Lemma \ref{lem:R_n-convergence}]
Let $U \equiv U^s$ be an open neighborhood of zero in $H^{s}_{c,0}$ so that the statement of 
Proposition \ref{prop:L-local} holds. Let $d\ge 2$ and any $1\le m\le d-1$. 
By \eqref{estimate l_j - mu}, one has
$$
| A(l_1,...,l_m | \le 5^m \prod_{j=1}^m \frac{1}{|l_j| + 1} , \qquad
| A(l_m,...,l_d | \le 5^{d -m+1} \prod_{j=m}^d \frac{1}{|l_j| + 1} .
$$
Hence $|A(l_1,...,l_m)| |A(l_m,...,l_d)| |E_u| \le 5^{d+1} F_u(l_1,...,l_m) F'_u(l_m,...,l_d )$ 
where
$$
F_u( l_1,...,l_m) := \frac{1}{|l_m| + 1} \frac{|\widehat u(l_m - l_{m-1})|}{|l_{m-1}| + 1} \cdots  
\frac{|\widehat u(l_2 - l_{1}|)}{|l_{1}| + 1} |\widehat u( l_{1})| ,
$$
$$
F_u'( l_m,...,l_d) :=  \frac{1}{| l_m | + 1} \frac{|\widehat u( - (l_{m} - l_{m+1}) )}{| l_{m+1}| + 1} \cdots  
\frac{|\widehat u( -(l_{d-1} - l_d))}{| l_{d}| + 1} |\widehat u( - l_d)| .
$$
By Lemma \ref{new version map Q} it then follows that for any $u \in U$
$$
\big\|\sum_{\tb{-\infty<l_j<\infty}{1\le j\le m-1}} F_u( l_1,...,l_m)\big\|_{s+1} \le C_s^{m-1} \|u\|_s^m, 
$$
$$
\big\|\sum_{\tb{-\infty<l_j<\infty}{m+ 1 \le j\le d}} F'_u( l_m,...,l_d)\big\|_{s+1} \le C_s^{d - m} \|u\|_s^{d-m +1}
$$
where $C_s \ge 1$ is the constant of Lemma \ref{new version map Q}. One then concludes that
the left hand side of the inequality in \eqref{eq:R_n-convergence} is bounded by 
$$
\sum_{d\ge 2} 5^{d+1} d^2 C_s^{d-1} \|u \|_s^{d+1} \le \sum_{d\ge 2} ( C_s' \|u\|_s)^{d+1}
$$
where the constant $C_s'$ satisfies $5^{d+1} d^2 C_s^{d-1} \le (C_s')^{d+1}$. By shrinking $U$, if needed, 
the claimed estimate \eqref{eq:R_n-convergence} follows.

\end{proof}

\begin{Lem}\label{lem:l^1-analyticity}
Let $U$ be an open neighborhood in a complex Banach space $X$. Assume that the map 
$f : U\to\ell^1_+$, $x\mapsto\big(f_n(x)\big)_{n\ge 1}$ is locally bounded and for any 
$n\ge 1$ the ``coordinate'' function $f_n : X\to\C$ is analytic. Then $f : X\to\ell^1_+$ is analytic.
\end{Lem}

\begin{proof}[Proof of Lemma \ref{lem:l^1-analyticity}]
The lemma follows if we prove that $f : U\to\ell^1_+$ is weakly analytic (cf. \cite[Appendix A]{KP-book}). 
To this end, 
fix $x_0\in U$, $h\in X$, $\|h\|_X=1$, and $c=(c_n)_{n\ge 1}\in\ell^\infty_+ \equiv\big(\ell^1_+\big)^*$.
For any $\lambda \in \C$ with $|\lambda|<r$ and $r>0$ sufficiently small, 
consider the complex valued function
\[
g(\lambda):= \sum_{n\ge 1}c_n f_n(x_0+\lambda h)
\]
and the partial sums $g_N(\lambda):=\sum_{n=1}^N c_n f_n(x_0+\lambda h)$, $N\ge 1$.
Since $f : U\to\ell^1_+$ is locally bounded, there exist $M>0$, $r>0$ so that 
for any $|\lambda|<r$ and  $N\ge 1$,
\[
|g_N(\lambda)|\le\|c\|_{\ell^\infty_+} \|g(x_0+ \lambda h)\|_{\ell^1_+} \le M.
\]
Hence, the holomorphic functions $(g_N)_{N\ge 1}$ are uniformly bounded in 
the disk $D_0(r) = \{\lambda\in\C \, \big| \, |\lambda|<r\}$. By Montel's theorem, there exists a subsequence $(g_{N_k})_{k\ge 1}$ 
of  $(g_N)_{N\ge 1}$, converging to $g$ uniformly on any compact subset of $D_0(r)$. Hence, by the Weierstrass theorem, 
$g$ is holomorphic. Since $x_0 \in U$, $h \in X$ with $\|h\|_X =1$, and $(c_n)_{n \ge 1} \in \ell^\infty_+$ 
are arbitrary, $f $ is weakly analytic.
\end{proof}

\section{Proof of Theorem \ref{th:Phi} and Theorem \ref{th:well-posedness}}\label{proof main results}
In this section we prove Theorem \ref{th:Phi} and Theorem \ref{th:well-posedness}, stated
in Section \ref{Introduction}.

\begin{proof}[Proof of Theorem \ref{th:Phi}]
Recall from Proposition \ref{prop:Phi-analytic} that for any $s>-1/2$, there exists a neighborhood $U \equiv U^s$ of zero in 
$H^s_{c,0}$ so that the Birkhoff map extends from $U\cap H^s_{r,0}$ to an analytic map  
\begin{equation}\label{eq:Phi-modified1}
\Phi : U\to\h^{\frac{1}{2}+s}_{c,0},\quad
u\mapsto\Big(\big(\Phi^*_{-n}(u)\big)_{n\le -1},\big(\Phi_n(u)\big)_{n\ge 1}\Big)
\end{equation}
where, by \eqref{eq:Phi2},
\begin{equation}\label{eq:Phi2'}
\Phi_n(u)=\sqrt{n}\frac{a_n^*(u)}{\sqrt[+]{n\kappa_n(u)}}\Psi^*_n(u),\quad n\ge 1,
\end{equation}
with $\Psi : U\to\h^{1+s}_+$ being the pre-Birkhoff map  \eqref{eq:Psi}, studied in 
Section \ref{sec:Psi}. 
Let us compute the differential $d_0\Phi$ of $\Phi$ at $u=0$. To this end note that
the differential of $\Psi$ at $u=0$ is given by the weighted 
Fourier transform (cf. \eqref{eq:Psi-linear_term}),
$d_0\Psi : H^s_{c,0}\to\h^{1+s}_+$, $u\mapsto\left(-\frac{\hu(-n)}{n}\right)_{n\ge 1}$.
Hence, in view of \eqref{eq:F*},
\begin{equation}
d_0\Psi^* : H^s_{c,0}\to\h^{1+s}_+,\quad u\mapsto\left(-\frac{\hu(n)}{n}\right)_{n\ge 1}\,.
\end{equation}
Recall from Remark \ref{rem:normalizations} that for $u=0$ one has
\begin{equation}\label{eq:kappa_at_zero}
\kappa_0(0)= 1 , \ \  n\kappa_n(0)=1, \ n\ge 1, \qquad
f_n(0) = e_n, \ \ h_n(0)=e_n, \ \ \ n\ge 0.
\end{equation}
The definition \eqref{eq:f<->h} of $a_n(u)$ then implies that for $u=0$,
\begin{equation}\label{eq:a_at_zero}
a^*_n(0)=\overline{a_n(0)}=1,\qquad n\ge 1.
\end{equation}
From \eqref{eq:Phi-modified1}--\eqref{eq:a_at_zero}, the fact that $\Psi_n(u)=0$, $n\ge 1$, and Leibniz's rule 
one then infers that
\begin{equation}\label{eq:d_Phi}
d_0\Phi : H^s_{c,0}\to\h^{\frac{1}{2}+s}_{c,0},\quad
u\mapsto\left(-\frac{\hu(n)}{\sqrt{|n|}}\right)_{n\in\Z\setminus\{0\}}\,.
\end{equation}
Hence,  being a weighted Fourier transform, $d_0\Phi : H^s_{c,0}\to\h^{\frac{1}{2}+s}_{c,0}$ 
is a linear isomorphism of complex linear spaces.
Moreover, $d_0\Phi\big|_{H^s_{r,0}}  : H^s_{r,0}\to\h^{\frac{1}{2}+s}_{r,0}$ is a linear 
isomorphism of the corresponding real subspaces. Theorem \ref{th:Phi} now follows from the inverse function theorem 
in Banach spaces together with the fact that $\Phi\big|_{H^s_{r,0}} : H^s_{r,0}\to\h^{\frac{1}{2}+s}_{r,0}$
is a homeomorphism (cf.  \cite[Theorem 6]{GKT1}).
\end{proof}

Let us now turn to the proof of Theorem \ref{th:well-posedness}.  To prove item (i), we first need to make some preliminary considerations.
Recall that  the Benjamin-Ono equation is globally well-posed in $H^s_{r,0}$ for any 
$s > -1/2$ (cf. \cite{GKT1} for details and references). For any given $t \in \mathbb R$,  denote by $\mathcal S^t_0$ the 
flow map of the Benjamin-Ono equation on $H^s_{r,0}$, $\mathcal S^t_0 : H^s_{r,0} \to H^s_{r,0}$ and by 
$\mathcal S^t_{B}: \h^{1/2+s}_+ \to \h^{1/2+s}_+$ the version of $\mathcal S^t_0$ obtained, when expressed in 
the Birkhoff coordinates $(\zeta_n)_{n \ge 1}$ (cf. Remark \ref{rem:Birkhoff_maps}). 
To describe $\mathcal S^t_{B}$ more explicitly, recall that the nth frequency $\omega_n$, $n \ge 1$, of the Benjamin-Ono 
equation is the real valued, affine function defined on $\h^{s+1/2}_+$ (cf. \cite{GK}, \cite{GKT1}),
\begin{equation}\label{eq:formula_nth_frequency}
\omega_n(\zeta) = n^2 -2\sum_{k=1}^{n} k |\zeta_k|^2
- 2n\sum_{k=n+1}^{\infty} |\zeta_k|^2\, .
\end{equation}
For any initial data  $\zeta(0) \in \h^{1/2 + s}_+$, $ s > -1/2$, $\mathcal S^t_B(\zeta(0))$ is given by
\begin{equation}\label{formula for S_B}
\mathcal S^t_B(\zeta(0)):= \big( \zeta_n(0)e^{it\omega_n(\zeta(0))} \big)_{n \ge 1} \, .
\end{equation}
The key ingredient into the proof of Theorem \ref{th:well-posedness} (i)
is a corresponding result for the flow map $\mathcal S^t_B$. 
More precisely, one has the following
\begin{Lem}\label{lem:nowhere locally uniformly continuous}
For any $t \ne 0$ and any $ - 1/2 < s < 0$, 
$\mathcal S^t_B : \h^{1/2+s}_+ \to \h^{1/2+s}_+$
is nowhere locally uniformly continuous. In particular,
it is {\em not} locally Lipschitz.
\end{Lem}
\begin{proof}[Proof of Lemma \ref{lem:nowhere locally uniformly continuous}]
We argue as in the proof of a corresponding result for the KdV equation in \cite[Theorem 3.10]{KM1}.
Let $U \equiv U^s$ be an arbitrary non-empty open subset of $\h^{1/2+s}_+$
with $-1/2  < s < 0$ and $t \ne 0$.
Choose $\zeta^{(0)} \in U$ so that there exists $N\ge 1$
with the property that $\zeta^{(0)}_n = 0$ for any
$n > N$. For any $\delta > 0$ and $m > N,$
let
$\zeta^{(m, \delta)} : = (\zeta_n^{(m, \delta)})_{n \ge 1}$ and
$\xi^{(m, \delta)} : = (\xi_n^{(m, \delta)})_{n \ge 1}$,
where
$$
\zeta_n^{(m, \delta)} = \zeta_n^{(0)}\, , \quad
\forall n \ne m, 
\qquad \zeta_m^{(m, \delta)} = \frac{\delta}{m^{1/2 + s}}\, ,  \qquad \quad
$$ 
and 
$$
\xi_n^{(m, \delta)} = \zeta_n^{(0)}\, , \quad
\forall n \ne m, \qquad 
\xi_m^{(m, \delta)} = \frac{\delta(1 + i m^{ s/2})}{m^{1/2 + s}} \,  . 
$$
Then
$$
\big\|\zeta^{(m, \delta)}-\zeta^{(0)}\big\|_{\h^{1/2 + s}_+} 
=\delta\, , \qquad  
\big\|\xi^{(m, \delta)}-\zeta^{(0)}\big\|_{\h^{1/2 + s}_+} 
=\delta\sqrt{1+m^{s}} \, .
$$
Choose $\delta_0 > 0$ so small that 
$\zeta^{(m, \delta)}$ and 
$\xi^{(m, \delta)}$ are elements of $U$
for any $m > N$ and $0 < \delta \le  \delta_0$.
Furthermore one has for any $m > N,$
\begin{equation}\label{difference initial data}
\big\|\zeta^{(m, \delta)}-\xi^{(m, \delta)}\big\|_{\h^{1/2 + s}_+}
=\delta\, m^{s/2}\, , 
\end{equation}
and by \eqref{eq:formula_nth_frequency},
$$
\big|\omega_m( {\zeta^{(m, \delta)}})-\omega_m( {\xi^{(m, \delta)}} )\big| = 
2m \frac{\delta^2 m^{s}}{m^{1+2s}}=\frac{1}{m^s} 2\delta^2\, .
$$
For any given $t \ne 0$, choose an integer $k \ge 1$ so large that 
$$
\delta \equiv \delta(t):=\left(\frac{\pi k^s}{2|t|}\right)^{1/2}  <  \delta_0 \, 
$$
and hence $2|t|\delta^2=\pi k^{s}$. It then follows that
$$
\big|\omega_m( {\zeta^{(m, \delta)}} ) t -\omega_m( {\xi^{(m, \delta)}} ) t\big|
=\left(\frac{ k}{m}\right)^{s}\pi .
$$
Since $0 < |s| < 1/2$ and
$$
\big|(m + 1)^{|s|} - m^{|s|}\big| \le |s| \int_m^{m+1} 
\frac{1}{x^{1 -|s|}} dx \le |s| ,
$$
there exists a subsequence $(m_j)_{j \ge 1}$ 
of the sequence $(k n)_{n \ge N+1}$ so that
(with $\mathbb N = \{ 1, 2, 3, \ldots  \}$),
$$
\text{\rm dist}\left(\left(\frac{k}{m_j}\right)^{s}, \, 2\mathbb N-1\right)<1/2,\qquad \forall j \ge 1\, .
$$
Since for any $x \in \R$ with ${\rm{dist}}( x,\, 2\mathbb N - 1) < 1/2$
one has $\big|\exp(\pm ix\pi)-1\big|>1$, one concludes that
$$
\big|\exp\big(i t\omega_{m_j}({\zeta^{(m_j, \delta)}})-i t\omega_{m_j}({\xi^{(m_j,\delta)}})\big)-1\big|>1
$$
and hence
\begin{align}
& m_j^{1/2 + s} | \xi_{m_j}^{(m_j, \delta)} | \cdot\big| \exp\big( it \omega_{m_j}( \zeta^{(m_j, \delta)} )\big) -
\exp\big( it \omega_{m_j}( \xi^{(m_j, \delta) } ) \big)\big|\nonumber\\
& = m_j^{1/2 + s}
|\xi_{m_j}^{(m_j, \delta)}|\cdot\big|\exp\Big( it\big(\omega_{m_j}( {\zeta^{(m_j, \delta)}} ) - 
\omega_{m_j}( {\xi^{(m_j, \delta)}} ) \big) \Big) - 1\big|   \nonumber  \\
& \ge \sqrt{ 1 + m_j^{s}} \, \delta  \, . \label{estimate 1}
\end{align}
In view of the estimate (cf. \eqref{difference initial data})
\begin{equation}\label{estimate 2}
m_j^{1/2 + s}\big|\zeta_{m_j}^{(m_j, \delta)} - \xi_{m_j}^{(m_j, \delta)}\big|\cdot  
\big|\exp\big( it \omega_{m_j}( \zeta^{(m_j, \delta) }) \big)\big| \le \delta m_j^{ s/2} ,
\end{equation}
one then concludes, by comparing only the $m_j$th component of $\mathcal S_B^t(\zeta^{(m_j, \delta)})$
with the one of  $ \mathcal S_B^t(\xi^{(m_j, \delta)})$ (cf. \eqref{formula for S_B}) that
\begin{align}
\big\|\mathcal S_B^t & (\zeta^{(m_j, \delta)})-\mathcal S_B^t(\xi^{(m_j,\delta)})\big\|_{\h^{1/2 + s}_+}\ge 
m_j^{1/2 + s}\big|\zeta_{m_j}^{(m_j, \delta)}  (t) - \xi_{m_j}^{(m_j,\delta)} (t)\big|\nonumber\\
& \ge  m_j^{1/2 + s} | \xi_{m_j}^{(m_j, \delta)} | \cdot\Big|\exp\big( it \omega_{m_j}( \zeta^{(m_j, \delta)} )\big) -
\exp\big( it \omega_{m_j}( \xi^{(m_j, \delta) } ) \big)\Big|\nonumber\\
&-m_j^{1/2+s}\big|\zeta_{m_j}^{(m_j, \delta)}-\xi_{m_j}^{(m_j,\delta)}\big|\,.\nonumber
\end{align}
This together with \eqref{estimate 1} and \eqref{estimate 2} yields
$$
\big\|\mathcal S_B^t  (\zeta^{(m_j, \delta)})-\mathcal S_B^t(\xi^{(m_j, \delta)})\big\|_{\h^{1/2 + s}_+}\ge  
\big((1+m_j^{s})^{1/2}-m_j^{s/2}\big)\delta\,,
$$
with the latter expression converging to $\delta>0$ as $j\to\infty$, whereas by \eqref{difference initial data},
$$
\big\|\zeta^{(m_j,\delta)}-\xi^{(m_j,\delta)}\big\|_{\h^{1/2 + s}_+}
=\delta\, m_j^{s/2}\, , 
$$
which converges to $0$ as  $j\to\infty$. This completes the proof of Lemma \ref{lem:nowhere locally uniformly continuous}.
\end{proof}

\begin{proof}[Proof of Theorem \ref{th:well-posedness}(i)]
The claimed result follows directly from Lemma \ref{lem:nowhere locally uniformly continuous} and 
Theorem \ref{th:Phi}. 
\end{proof}

\smallskip

To prove Theorem \ref{th:well-posedness}(ii), we first need to make some preliminary considerations. 
Denote by $\ell^\infty_{c,0}$ the space $\ell^\infty(\Z \setminus\{0\}, \C)$ 
of bounded complex valued sequences $z = (z_n)_{n \ne 0}$, endowed with the supremum norm. 
It is a Banach space and can be identified with the subspace of $\ell^\infty_c$, consisting of sequences $(z_n)_{n \in \Z}$ with $z_0 = 0$.
Note that  $\ell^\infty_{c,0}$ is a Banach algebra with respect to the multiplication
\begin{equation}\label{eq:l^infty-Banach_algebra}
\ell^\infty_{c,0}\times\ell^\infty_{c,0}\to\ell^\infty_{c,0},\quad(z,w)\mapsto z\bcdot w:=(z_n w_n)_{n\ne 0}.
\end{equation}
Similarly, for any $s\in\R$, the bilinear map
\begin{equation}\label{eq:h*l^infty-multiplication}
\h^s_{c,0}\times\ell^\infty_{c,0}\to\h^s_{c,0},\quad(z,w)\mapsto z\bcdot w,
\end{equation}
is bounded. For any given $T>0$ and $s\in\R$ introduce the $\C$-Banach spaces
\[
\mathcal{C}_{T, s}:=C\big([-T,T],\h^s_{c,0}\big)\quad\text{\rm and}\quad
\mathcal{C}_{T, \infty} :=C\big([-T,T],\ell^\infty_{c,0}\big),
\] 
endowed with supremum norm.
Elements of these spaces are denoted by $\xi$, or more explicitly, $\xi(t) = (\xi_n(t))_{n \ne 0}$.
The multiplication in \eqref{eq:l^infty-Banach_algebra} and \eqref{eq:h*l^infty-multiplication} induces
in a natural way the following bounded bilinear maps
\begin{equation}\label{eq:multiplication_on_curves}
\mathcal{C}_{T, \infty} \times\mathcal{C}_{T, \infty} \to\mathcal{C}_{T, \infty},\qquad
\mathcal{C}_{T, s} \times\mathcal{C}_{T, \infty}\to\mathcal{C}_{T, s},  \qquad
\mathcal{C}_{T,s} \times \ell^\infty_{c,0}\to\mathcal{C}_{T,s}.
\end{equation}
For example, the boundedness of the first map in \eqref{eq:multiplication_on_curves} follows from 
the Banach algebra property of the multiplication in $\ell^\infty_{c,0}$ and the estimate
\begin{align*}
\|\xi^{(1)}\bcdot\xi^{(2)}\|_{\mathcal C_{T,\infty}}&=
\sup_{t\in[-T,T]}\|\xi^{(1)}(t)\bcdot\xi^{(2)}(t)\|_{\ell_c^\infty}\le
\sup_{t\in[-T,T]}\|\xi^{(1)}(t)\|_{\ell_c^\infty}\|\xi^{(2)}(t)\|_{\ell_c^\infty}\\
&\le\|\xi^{(1)}\|_{\mathcal C_{T,\infty}}\|\xi^{(2)}\|_{\mathcal C_{T,\infty}},\qquad \xi^{(1)},\xi^{(2)}\in \mathcal C_{T,\infty}.
\end{align*}
Hence, $\big(\mathcal{C}_{T, \infty},\bcdot\big)$ is a Banach algebra.
The following lemma can be shown in a straightforward way and hence we omit its proof.

\begin{Lem}\label{lem:exp-analytic}
The map 
$\mathcal{C}_{T,\infty} \to \mathcal{C}_{T,\infty}$, 
$\xi \mapsto 
e^\xi:=\big(e^{\xi_n}\big)_{n\ne 0}$,
is analytic.
\end{Lem}

For any $s>-1/2$ and any $n\ge 1$, the $n$th frequencies \eqref{eq:formula_nth_frequency} of the
Benjamin-Ono equation extends from $\h_{r,0}^{\frac{1}{2}+s}$ (cf. Remark \ref{rem:Birkhoff_maps}) to 
an analytic function on $\h_{c,0}^{\frac{1}{2}+s}$ given by
\begin{equation}\label{eq:omega_n}
\omega_n(\zeta) = n^2+\Omega_n(\zeta),
\end{equation}
where 
\begin{equation}\label{eq:Omega_n}
\Omega_n(\zeta):=-2\sum_{k=1}^{n} k\,\zeta_{-k}\zeta_k
- 2n\sum_{k=n+1}^\infty\zeta_{-k}\zeta_k,\qquad\zeta\in\h_{c,0}^{\frac{1}{2}+s}.
\end{equation}
For $n\le-1$ we set 
\begin{equation}\label{eq:Omega_{-n}}
\Omega_n(\zeta):=-\Omega_{-n}(\zeta),\quad\zeta\in\h_{c,0}^{\frac{1}{2}+s}.
\end{equation}
We have the following

\begin{Lem}\label{eq:Omega-analytic}
For any $s\ge 0$ the map
\begin{equation}\label{eq:Omega-map}
\Omega : \h_{c,0}^{\frac{1}{2}+s}\to\ell^\infty_{c,0},\quad\zeta\mapsto\big(\Omega_n(\zeta)\big)_{n\ne 0},
\end{equation}
is analytic.
\end{Lem}
\begin{proof}[Proof of Lemma \ref{eq:Omega-analytic}] 
Let $s\ge 0$.
Then for any $n \ge 1$,  $\zeta\in\h^{\frac{1}{2}+s}_{c,0}$, one has by \eqref{eq:Omega_n}
\begin{align*}
\big|\Omega_n(\zeta)\big|&\le 2\,\Big|\sum_{k=1}^{n} k\,\zeta_{-k}\zeta_k\Big|
+2\,\Big|n\sum_{k=n+1}^\infty\zeta_{-k}\zeta_k\Big|
\le2\,\sum_{k=1}^\infty|k||\zeta_{-k}||\zeta_k|\le 2\,\|\zeta\|_{\h^{1/2+s}_c}^2.
\end{align*}
This, together with \eqref{eq:Omega_{-n}} then implies that for any $s\ge 0$, the map \eqref{eq:Omega-map} is 
well defined and locally bounded. The estimate above also shows that for any $n\ne 0$ the $n$th component of 
$\eqref{eq:Omega-map}$ is a (continuous) quadratic form in $\zeta\in\h^{\frac{1}{2}+s}_{c,0}$
and hence analytic. The lemma then follows from \cite[Theorem A.3]{KP-book}.
\end{proof}

We now consider the curves $\xi^{(1)},\xi^{(2)}\in\mathcal{C}_{T,\infty}$, defined by
\begin{equation}\label{eq:gamma1}
\xi^{(1)} : [-T,T]\to\ell^\infty_{c,0},\qquad \xi_n^{(1)}(t):= e^{i \sign(n)\,n^2 t} , \ \ n\ne 0,
\end{equation}
and, respectively, 
\begin{equation}\label{eq:gamma2}
\xi^{(2)} : [-T,T]\to\ell^\infty_{c,0},\qquad \xi_n^{(2)}(t):=t , \ \ n\ne 0.
\end{equation}
The following corollary follows from Lemma \ref{lem:exp-analytic}, Lemma \ref{eq:Omega-analytic}, 
and the continuity of the bilinear maps in \eqref{eq:multiplication_on_curves}.

\begin{Coro}\label{coro:analyticity_of_S_B}
For any $s\ge 0$, the map (cf. \eqref{formula for S_B})
\begin{equation}\label{eq:S_{B,T}}
\mathcal{S}_{B,T} : \h^{\frac{1}{2}+s}_{c,0}\to\mathcal{C}_{T, \frac{1}{2}+s},\quad
\zeta\mapsto\zeta\bcdot \xi^{(1)}\bcdot e^{i\,\Omega(\zeta)\bcdot\xi^{(2)}},
\end{equation}
is analytic.
\end{Coro}

\begin{proof}[Proof of Theorem \ref{th:well-posedness}(ii)] 
Let $s \ge 0$ and let $U^s$ be the neighborhood of zero in $H^s_{c,0}$ of Theorem \ref{th:Phi},
so that $\Phi: U^s \to \Phi(U^s) \subseteq\h^{\frac{1}{2}+s}_{c,0}$ is a diffeomorphism.
According to Corollary \ref{coro:analyticity_of_S_B} there exists a neighborhood $W$ of zero in 
$\h^{\frac{1}{2}+s}_{c,0}$, which is contained in $\Phi(U^s)$, so that for any $\zeta \in W$, 
one has $\mathcal S^t_B(\zeta) \in \Phi(U^s)$ for any $t \in [-T, T]$. Let $V^s := \Phi^{-1}(W)$.
By Corollary \ref{coro:analyticity_of_S_B}, Theorem \ref{th:Phi},
Lemma \ref{lem:push-forward} below, and the fact $\mathcal{S}_0^t(u_0)=( \Phi^{-1}\circ S^t_B\circ\Phi)(u_0)$
for any $u_0\in V^s$ and $t\in[-T,T]$, it then follows that 
$\mathcal S_{0, T}$ extends to an analytic map $V^s \to C\big([-T, T], H^s_{c, 0}\big)$.
\end{proof}

The following lemma is used in the proof of Theorem \ref{th:well-posedness}(ii).
To state it, we first need to introduce some more notation.
Let $X$ be a complex Banach space with norm $\| \cdot \| \equiv \|\cdot \|_X$. For any $T > 0$, denote by $C\big([-T,T],X\big)$
the Banach space of continuous functions $x: [-T, T] \to X$, endowed with the supremum norm $\|x\|_{T, X} := \sup_{t \in [-T, T]} \| x(t) \|$. 
Furthermore, for any open neighborhood $U$ of $X$, denote by 
$C\big([-T,T],U\big)$ the subset of $C\big([-T,T], X\big)$, consisting of continuous functions $[-T, T ] \to X$ with values in $U$. 
\begin{Lem}\label{lem:push-forward}
Let $f : U\to Y$ be an analytic map from an open neighborhood  $U$ in $X$ where 
$X$ and $Y$ are complex Banach spaces. Then for any $T > 0$, the associated push forward map
\[
f_\star : C\big([-T,T],U\big)\to C\big([-T,T],Y\big),\quad [ t \mapsto x(t)]  \mapsto\big[t\mapsto f(x(t))\big],
\]
is analytic. 
\end{Lem}

\begin{proof}[Proof of Lemma \ref{lem:push-forward}]
Let $x_0\in C\big([-T,T],U\big)$ be given.
Since $x_0\big([-T,T]\big)$ is compact in $U$ and 
$f : U \to Y$ is locally bounded there exist $M>0$ and $r>0$ so that for any $x \in C\big([-T,T], U\big)$ with 
$\|x-x_0\|_{T, X} < 2 r$ one has
\begin{equation}\label{eq:f_*-boundedness}
\| f_*x\|_{T, Y} := \sup_{t \in [-T, T]} \|f(x(t))\|_Y \le M, 
\end{equation}
where  $\| \cdot \|_Y$ denotes the norm of $Y$.
For any given $t\in[-T,T]$ we obtain from the analyticity of the map $f : U\to Y$ and Cauchy's formula that 
for any $x$, $h\in C\big([-T,T],U\big)$ such that 
\begin{equation}\label{eq:gamma,h-uniformity_condition}
\|x- x_0\|_{T, X}<r/2, \qquad  \|h\|_{T, X}<r,
\end{equation}
we have that for any $0<|\mu|<1/2$,
\begin{align*}
\frac{f\big(x(t)+\mu h(t)\big)-f(x(t))}{\mu}=
\frac{1}{2\pi i}\oint_{|\lambda|=1}\frac{f\big(x(t)+\lambda h(t)\big)}{\lambda(\lambda-\mu)}\,d\lambda
\end{align*}
and
\[
d_{x(t)}f\big(h(t)\big)=\frac{1}{2\pi i}\oint_{|\lambda|=1}
\frac{f\big(x(t)+\lambda h(t)\big)}{\lambda^2}\,d\lambda
\]
These identities imply that for any $x$, $h\in C\big([-T,T],U\big)$ satisfying \eqref{eq:gamma,h-uniformity_condition},
and for any $t\in[-T,T]$, $0<|\mu|<1/2$, one has
\begin{align*}
\left\|\frac{f\big(x(t)+ \mu h(t)\big)-f(x(t))}{\mu}-d_{x(t)}f\big(h(t)\big)\right\|_Y&\!\!\!\!\!=
\left\|\frac{\mu}{2\pi i}\oint_{|\lambda|=1}\!\!\!\!\frac{f\big(x(t)+\mu h(t)\big)}{\lambda^2(\lambda-\mu)} d \lambda \right\|_Y\\
&\le 2M |\mu |
\end{align*}
where we used \eqref{eq:f_*-boundedness}. Expressed in terms of the push forward map $f_*$, it means that for any 
$x$, $h\in C\big([-T,T],U\big)$ satisfying \eqref{eq:gamma,h-uniformity_condition}, 
\begin{equation*}
\left\|\frac{f_*\big(x+ \mu h\big)-f_*(x)}{\mu}- d_{x}f_*(h)\right\|_{T, Y}\le 2M|\mu| , \qquad 0<|\mu|<1/2 .
\end{equation*}
Therefore,
$f_* : C\big([-T,T],U\big)\to C\big([-T,T],Y\big)$ is a $C^1$-map between $\C$-Banach spaces (see e.g. \cite[Problem 3]{PTrub})
and hence analytic.
\end{proof}

It remains to prove the Addendum to Theorem \ref{th:well-posedness} (ii), stated in Section \ref{Introduction}.

\begin{proof}[Proof of Addendum to Theorem \ref{th:well-posedness}(ii)] 
Corollary \ref{coro:analyticity_of_S_B} continues to hold when $\mathcal{C}_{T, {\frac{1}{2}+s}}$ is replaced by 
the Banach space
$C^k\big([-T,T], \h^{\frac{1}{2}+s-2k}_{c,0}\big)$, $k\ge 1$, of $k$ times continuously differentiable functions 
$\xi : [-T, T] \to  \h^{\frac{1}{2}+s-2k}_{c,0}$. For any $k \ge 1$, this follows from the proof of 
Theorem \ref{th:well-posedness}(ii) and the fact that the curve $\zeta\bcdot\xi^{(1)}$ 
[resp. $e^{i\,\Omega(\zeta)\bcdot\xi^{(2)}}$], which appears as a factor on the right side of 
\eqref{eq:S_{B,T}}, belongs to $C^k\big([-T,T],\h^{\frac{1}{2}+s-2k}_{c,0}\big)$ 
[resp. $C^k\big([-T,T],\ell^\infty_{c,0}\big)$]. 
By combining this with Theorem \ref{th:Phi} one obtains the claimed result.
\end{proof}

\appendix

\section{Symmetries of the Lax operator}\label{Appendix symmetries}
In this appendix we study symmetries of the Lax operator $L_u$. 
Recall that in \cite{GKT2}, we defined and studied the Lax operator 
$L_u = D - T_u: H^{1+s}_+  \to H^s_+$ for $u\in H^{s}_c$, $-1/2 < s \le 0$. 
In a similar way one defines
\begin{equation}\label{eq:L^-}
L_u^-:= - D-T_u^- :  H^{1+s}_-   \to  H^s_- 
\end{equation}
where $T^-_u: H^{1+s}_-\to H^{s}_-$ is the Toeplitz operator with potential $u$, 
\begin{equation*}
T^-_u f:=\Pi^-(u f),\quad f\in H^{1+s}_-,
\end{equation*}
and $\Pi^- : H^{s}_c \to H^{s}_-$ is the Szeg\H o projector
\[
\Pi^- : H^{s}_c\to H^{s}_-,\quad \sum_{n\in\Z}\widehat v(n) e^{i n x}\mapsto\sum_{n\le 0}\widehat v(n) e^{i n x},
\]
onto the (negative) Hardy space $H^{s}_-=\big\{f\in H^{s}_c\, \big| \,\hat{f}(k)=0\,\,\,\forall k>0\big\}$.
Indeed, it follows from Lemma 1 in \cite{GKT2} that  $T^-_u: H^{1+s}_-\to H^{s}_-$ is a well defined bounded linear operator 
and hence $L_u^-$ defines an operator on $H^{s}_-$ with domain
$H^{1+s}_-$ so that the map $L_u^- : H^{1+s}_-\to H^{s}_-$ is bounded.
For $\beta \in \R$, denote by $(\, \cdot \, )_* : H^{\beta}_c\to H^{\beta}_c$ the involution, defined for $v\in H^{\beta}_c$ by
\begin{equation}\label{eq:*}
v_*(x):=v(-x),\quad x\in\T .
\end{equation}
Clearly, $(\, \cdot \, )_*$ is a $\C$-linear isometry. The Fourier coefficients of $v_*$ satisfy
\begin{equation}\label{eq:*_Fourier}
\widehat{{v}}_*(k)=\widehat{v}(-k),\quad k\in\Z .
\end{equation}

\begin{Lem}\label{lem:symmetries}
For any $u\in H^{s}_c$, $-1/2 < s \le 1/2$, we have the commutative diagrams
\[
\begin{tikzcd}
H^{1+s}_+\arrow[r, "L_u"]\arrow[d, swap, "(\cdot)_*"]&H^{s}_+\arrow[d, "(\cdot)_*"]\\
H^{1+s}_-\arrow[r, "L_{u_*}^-"]&H^{s}_-
\end{tikzcd}
\quad\quad\text{\rm and}\quad\quad
\begin{tikzcd}
H^{1+s}_+\arrow[r, "L_u"]\arrow[d, swap, "\overline{(\cdot)}"]&H^{s}_+\arrow[d, "\overline{(\cdot)}"]\\
H^{1+s}_-\arrow[r, "L_{\overline{u}}^-"]&H^{s}_-
\end{tikzcd}
\]
where $(\cdot)_*$ denotes the involution defined by \eqref{eq:*} and $\overline{(\cdot)} : H^\beta_+\to H^\beta_-$, 
$\beta\in\R$, is the complex conjugation of functions.
\end{Lem}

\begin{proof}[Proof of Lemma \ref{lem:symmetries}]
Since the proof of the commutativity of the second diagram is similar to the proof of the one
of the first diagram we will prove only the first one.
For any $u\in H^{s}_c$ and $f\in H^{1+s}_+$ we have
\begin{align*}
\big(\Pi(u f)\big)_* &=\sum_{n\ge 0}\Big(\sum_{k\ge 0}\hu(n-k)\hat{f}(k)\Big) e^{-i n x}
=\sum_{n\le 0}\Big(\sum_{k\le 0}\hu(k-n)\hat{f}(-k)\Big) e^{i n x}\\
& =\sum_{n\le 0}\Big(\sum_{k\le 0}\hu_*(n-k)\hat{f_*}(k)\Big) e^{i n x}=\Pi^-(u_* f_*).
\end{align*}
By combining this with the fact that $(Df)_*= - D(f_*)$ we conclude that 
\[
(L_uf)_*=L_{u_*}f_*.
\]
This completes the proof of the lemma.
\end{proof}

By combining Lemma \ref{lem:symmetries} with Proposition \ref{prop:L-global} we obtain

\begin{Coro}
For any $-1/2 < s  \le 0$, there exists an open neighborhood $W\equiv W^{s}$ of $H^{s}_{r,0}$ in $H^{s}_{c,0}$,
invariant under the involution \eqref{eq:*} and the complex conjugation of functions,
so that for any $u\in W$, the operator $L_u^-$ given by \eqref{eq:L^-} is a closed operator
in $H^{s}_-$ with domain $H^{1+s}_-$. The operator has a compact resolvent and
all its eigenvalues are  simple. When appropriately listed,  $\lambda^-_n$, $n\ge 0$, 
satisfy $\re(\lambda^-_n)<\re(\lambda^-_{n+1})$ for $n \ge 0$
and $|\lambda_n^-(u)-n|\to 0$ as $n\to\infty$. 
\end{Coro}

For $u\in H^{s}_{r,0}$, $ -1/2 < s \le 0$, let $(f_n(u))_{n\ge 0}$, be the eigenfunctions of $L_u$,
corresponding to the eigenvalues $(\lambda_n(u))_{n \ge 0}$, normalized 
as
\begin{equation}\label{eq:f_n-normalization}
\|f_n(u)\|=1, \ n\ge 0, \qquad 
\1 1|f_0(u)\2>0,\quad\big\1 f_n(u)|S f_{n-1}(u)\big\2>0,\ n\ge 1
\end{equation}
(cf. \cite[Definition 2.1]{GK})). Similarly, denote by $(f_n^-(u))_{n\ge 0}$, the eigenfunctions of $L_u^-$,
corresponding to the eigenvalues $(\lambda^-_n(u))_{n \ge 0}$, normalized as
\begin{equation}\label{eq:f_n^{-}-normalization}
\|f^-_n(u)\|=1, \ n\ge 0, \quad 
\1 1|f^-_0(u)\2>0,\quad \big\1 Sf^-_n(u)|f^-_{n-1}(u)\big\2>0,\ n\ge 1.
\end{equation}
It follows from Proposition \ref{prop:L-local} (cf. \eqref{P_n anal near 0}) and Lemma \ref{lem:symmetries}
that for any $-1/2 < s \le 0$, there exists 
an open neighborhood $ U^s$ of zero in $H^{s}_{c,0}$ so that for any $u\in U^s$ and $n\ge 0$,
the Riesz projector
\begin{equation*}
P^-_n(u)=
-\frac{1}{2\pi i}\oint\limits_{\partial D_n}(L^-_u-\lambda)^{-1}\,d\lambda\in\LL\big(H^{s}_-,H^{1+s}_-\big),
\end{equation*}
is well defined and  the map $U^s \to \LL\big(H^{s}_-,H^{1+s}_-\big)$, $u\mapsto P^-_n(u)$, is analytic. 
Here $\partial D_n$ is the counterclockwise oriented boundary of $D_n\equiv D_n(1/3)$ (see \eqref{def:D_n}).
Hence for any $n\ge 0$, the map
\begin{equation}\label{eq:h_n^-}
U^s \to H^{1+s}_-,\quad u\mapsto h^-_n( u):=P^-_n(u) e_{-n}\in H^{1+s}_- .
\end{equation}
is analytic. Since the map $(\, \cdot \, )_*$ and the complex conjugation are isometries, 
without of loss of generality, we can choose $U^s$ to be invariant under these two maps.
With this notation established, we can state the following

\begin{Coro}\label{coro:symmetries}
Let $-1/2 < s \le 0$.
Then for $u\in H^{s}_{r,0}$ and $n \ge 0$,
\begin{equation}\label{eq:lambda_n+symmetries}
\lambda_n^-(u)=\overline{\lambda_n(\overline{u})}=\lambda_n(u_*), 
\qquad
\overline{\lambda_n(\overline{u})}= \lambda_n(u),
\end{equation}
and
\begin{equation}\label{eq:f_n+symmetries}
f_n^-(u)=\overline{f_n(\overline{u})}=\big(f_n(u_*)\big)_*.
\end{equation}
Similarly, for any $u\in H^{s}_{r,0} \cap U^s$ and $n \ge 0$,
\begin{equation}\label{eq:h_n+symmetries}
h_n^-(u)=\overline{h_n(\overline{u})}=\big(h_n(u_*)\big)_*
\end{equation}
where $h_n(u)$ is given by \eqref{eq:h_n} and $h_n^-(u)$ by \eqref{eq:h_n^-}.
\end{Coro}

\begin{Rem}
In a straightforward way it follows from Proposition \ref{prop:L-global}, Proposition \ref{prop:L-local}, and Remark \ref{rem:finite-gap} 
that for any $ -1/2 < s \le 0$ and $n \ge 0$, the identities 
$\lambda_n^-(u)=\lambda_n(u_*)$,
$\overline{\lambda_n(\overline{u})}=\lambda_n(u)$
(cf. \eqref{eq:lambda_n+symmetries}) and
$f_n^-(u)=\big(f_n(u_*)\big)_*$ (cf. \eqref{eq:f_n+symmetries})
extend by analyticity to a neighborhood $W^s$ of $H^{s}_{r,0}$ in $H^{s}_{c,0}$.
Similarly, the identity $h_n^-(u)= \big(h_n(u_*)\big)_*$ (cf. \eqref{eq:h_n+symmetries}) extends to the neighboordhood $U^s$.
\end{Rem}

\begin{proof}[Proof of Corollary \ref{coro:symmetries}]
Let $u\in H^{s}_{r,0}$ with $-1/2 < s \le 0$ and $n \ge 0$.
The identities \eqref{eq:lambda_n+symmetries} and the fact that 
\begin{equation}\label{eq:eigenfunctions_of_L_u^-}
\overline{f_n(\overline{u})}\quad\text{\rm and}\quad\big(f_n(u_*)\big)_*
\end{equation}
are eigenfunctions of $L_u^-$ with eigenvalue $\lambda_n^-(u)$ follow directly from 
the commutative diagrams in Lemma \ref{lem:symmetries}. 
Since for any $f,g\in H^{1+s}_{c,0}$,
$$
\|f\|=\|\overline{f}\|=\|f_*\| , \qquad 
\overline{\1 1|f\2}=\1 1|\overline{f}\2,  \qquad
\1 1|f\2=\big\1 1|f_*\big\2,
$$
and
\[
\overline{\1 S f|g\2}=\1\overline{f}|S\overline{g}\2,\qquad
\1 S f|g\2=\1 f_*|S g_*\2 ,
\]
 the normalization conditions \eqref{eq:f_n-normalization} then imply that
the eigenfunctions in \eqref{eq:eigenfunctions_of_L_u^-} satisfy the normalization 
condition \eqref{eq:f_n^{-}-normalization}.
By the simplicity of the eigenvalue $\lambda_n^-(u)$ we then conclude that 
$f_n^-(u)=\overline{f_n(\overline{u})}$ and $f_n^-(u)=\big(f_n(u_*)\big)_*$.

The last statement of the corollary follows from the definition of $h_n(u)$ in \eqref{eq:h_n} and Lemma \ref{lem:symmetries}.
Indeed, for any $u\in H^{s}_{r,0}$ and any $\lambda$ in the resolvent set of $L_u$,   Lemma \ref{lem:symmetries} implies that
\[
(L_u-\lambda)^{-1}= \mathcal C^{-1}(L_{\overline{u}}^--\bar{\lambda})^{-1} \mathcal C \qquad \text{\rm and} \qquad
(L_u-\lambda)^{-1}=\mathcal I^{-1} (L_{u_*}^--\lambda)^{-1} \mathcal I ,
\]
where $\mathcal C : H^{s}_+\to H^{s}_-$ denotes the restriction of the complex conjugation of functions to $H^s_+$ and
$\mathcal I: H^{s}_+\to H^{s}_-$ the one of the map $(\, \cdot \, )_*$.
This, together with \eqref{eq:h_n} then implies that for any $u\in H^{s}_{r,0} \cap U^s$ and $n\ge 0$ we have
\begin{eqnarray}
h_n(u)&=&
-\frac{1}{2\pi i}\oint\limits_{\partial D_n}
\mathcal C^{-1} \big((L_{\overline{u}}^--\bar{\lambda})^{-1}\,e_{-n}\big)\,d\lambda\nonumber\\
&=& \mathcal C^{-1}\Big( -\frac{1}{2\pi i}\oint\limits_{\partial D_n}
(L_{\overline{u}}^--\lambda)^{-1}\,e_{-n}\,d\lambda\Big)=\overline{h_n^-(\overline{u})}\label{eq:h_n+symmetries1} \, ,
\end{eqnarray}
where we used the definition \eqref{eq:h_n^-} of $h_n^-$.
By similar arguments one shows that  for any 
$u\in H^{s}_{r,0} \cap U^s$ and $n\ge 0$ we have
\begin{eqnarray}
h_n(u)&=& -\frac{1}{2\pi i}\oint\limits_{\partial D_n}
\mathcal I^{-1} \big((L_{u_*}^--\lambda)^{-1}\,e_{-n}\big)\,d\lambda\nonumber\\
&=& \mathcal I^{-1} \Big( - \frac{1}{2\pi i}\oint\limits_{\partial D_n}
(L_{u_*}^--\lambda)^{-1}\,e_{-n}\,d\lambda\Big)=\big(h_n^-(u_*)\big)_* , \label{eq:h_n+symmetries2}
\end{eqnarray}
where we used $\lambda_n^-(u)=\lambda_n(u_*)$ (cf. \eqref{eq:lambda_n+symmetries}), 
the assumption that $U^s$ is invariant under $\mathcal I$, and 
the definition \eqref{eq:h_n^-} of $h_n^-$.
The formulas in  \eqref{eq:h_n+symmetries} now follow from \eqref{eq:h_n+symmetries1} and \eqref{eq:h_n+symmetries2}.
\end{proof}

\end{document}